\documentclass[letterpaper]{article}

\usepackage{amsfonts}
\usepackage{amsmath}
\usepackage{latexsym}
\usepackage{amssymb}
\usepackage{graphicx}

\newtheorem{theorem}{Theorem}[section]
\newtheorem{corollary}[theorem]{Corollary}
\newtheorem{definition}[theorem]{Definition}
\newtheorem{lemma}[theorem]{Lemma}
\newenvironment{proof}[1][Proof]{\noindent\textbf{#1.} }{\ \rule{0.5em}{0.5em}}

\numberwithin{equation}{section}
\numberwithin{theorem}{section}

\newcommand{\lefthook}{\,\text{\raisebox{0.44ex}{\rule{4pt}{0.3pt}}}\!\overset{\lrcorner}{}\,}

\begin{document}

\begin{center}
{\LARGE $\vartheta$-functions on the Kodaira--Thurston manifold}
\\[0pt]\bigskip William D. Kirwin\footnote{\emph{Max Planck Institute for Mathematics in the Sciences,
Inselstrasse 22, 04109 Leipzig, Germany}.\\E-mail: kirwin@mis.mpg.de}
and Alejandro Uribe\footnote{\emph{Dept. of Mathematics, University of
Michigan, 2074 East Hall, 530 Church Street, Ann Arbor, MI 48109-1043}.\\E-mail: uribe@umich.edu}
\end{center}

\begin{abstract}
The Kodaira--Thurston $M$ manifold is a compact, $4$-dimensional nilmanifold which is symplectic and complex but not K\"{a}hler. We describe a construction of $\vartheta$-functions associated to $M$ which parallels  the classical theory of $\vartheta$-functions associated to the torus (from  the point of view of representation theory and geometry), and yields  pseudoperiodic complex-valued functions on $\mathbb{R}^4.$

There exists a three-step nilpotent Lie group $\widetilde{G}$ which acts transitively on the Kodaira--Thurston manifold $M$ in a Hamiltonian  fashion. The $\vartheta$-functions\ discussed in this paper are intimately  related to the representation theory of $\widetilde{G}$ in much the same  way the classical $\vartheta$-functions are related to the Heisenberg  group. One aspect of our results which has not appeared in the classical  theory is a connection between the representation theory of $\widetilde{G}$  and the existence of Lagrangian and special Lagrangian foliations and torus  fibrations in $M$.
\end{abstract}


\section{Introduction}

The classical theory of $\vartheta$-functions is a rich and beautiful  subject that weaves threads from a diverse set of mathematical disciplines.  It is the purpose of this note to describe a generalization of this theory  when viewed from a geometric/representation theoretic point of view. It is  the authors' hope that this generalization will not only illustrate  interesting new connections between $\vartheta$-functions and symplectic  geometry, but also clarify some aspects of the classical theory by  comparison.

We will develop a theory of $\vartheta$-functions associated to the  Kodaira--Thurston manifold, a certain nontrivial 2-torus bundle over a 2-torus, realized here as a compact nilmanifold, which is symplectic but \textit{not} K\"{a}hler. It seems that our constructions are  not unique to this situation and could be adapted to other compact nilmanifolds. Just as in the classical theory, $\vartheta$-functions  associated to the Kodaira--Thurston $M$ manifold arise when studying the  decomposition of the $L^2$-space of sections of certain line bundles over  $M$. The construction we give is intimately related to the symplectic  structure of $M$.

The main results of this paper are twofold. First, we give a construction  of $\vartheta$-functions on $M$ which parallels the classical theory, where possible. The construction we present of $\vartheta$-functions on $M$ uses the representation theory of an associated nilpotent Lie group $\widetilde{G} $, just as the classical $\vartheta$-functions are intimately related to the Heisenberg group (in fact, $\widetilde{G}$ can be interpreted as ``the Heisenberg group on  $\operatorname{Heis}(3)\times\mathbb{R}$'').

The second main result of this paper is a connection between the algebraic structure of $\widetilde{G}$ and the symplectic structure of $M$. To make  the construction of $\vartheta$-functions explicit requires a choice: a  subalgebra $\mathfrak{h}$ of $Lie(\widetilde{G})$ of a certain type (subordinate to a $4$-dimensional integral coadjoint orbit, to be precise). It turns out that $\mathfrak{h}$ is connected to the symplectic structure of $M$; we will see that each subordinate subalgebra $\mathfrak{h}$ corresponds to a Lagrangian foliation of $M$. If the subordinate subalgebra is an ideal, then the foliation is special Lagrangian. (Our proof of this fact is indirect; we enumerate all possible relevant subordinate subalgebras and observe that those which are ideals induce special  Lagrangian fibrations.) The family of such subordinate subalgebras can be  parameterized by $\mathbb{R}$, and in a certain parametrization, those foliations associated to the subalgebras corresponding to $0$ and $\pm\infty$ are torus fibrations.

In the remainder of this introduction, we will state our main results (though we leave some technical details for later). Next, we give an overview of the classical theory of $\vartheta$-functions so that the analogy of our results with the classical theory is apparent. We then briefly review the tool which we use to generalize $\vartheta$-functions to our situation: a generalization of geometric quantization to the symplectic category which is known as almost K\"{a}hler quantization. We conclude this section with a summary of the rest of the paper.

\subsection{Main results}

Let $G=\operatorname{Heis}(3)\times\mathbb{R}$ be the direct product of the  three-dimensional Heisenberg group with the real line. Denote by  $\Gamma_0$ the integer lattice in $G$. The \textit{Kodaira--Thurston manifold} is the compact quotient $M:=\Gamma_0\backslash G$. It can be equipped with a left $G$-invariant integral symplectic form $\omega$ and  complex structure (see Kodaira's work \cite{Kodaira}), but, as Thurston was the first to observe \cite{Thurston}, it is \emph{not} a K\"{a}hler  manifold; that is, the metric defined by any choice of complex structure and symplectic structure is not positive definite. In this paper, we will  primarily be interested in the symplectic structure of $M$.

Since we assume $\omega$ is integral, there is a Hermitian line bundle $\ell\rightarrow M$ with compatible connection whose curvature is the symplectic form. As we will see, there exists a central extension
\[
1\rightarrow\mathbb{R}\rightarrow\widetilde{G}\rightarrow G\rightarrow1
\]
such that $\widetilde{G}$ acts on $M$ in a Hamiltonian fashion. This Hamiltonian action lifts to the line bundle $\ell,$ and induces the \textit{right (quasi)regular representation} $\rho$ of $\widetilde{G}$ on $L^2(M,\ell),$ the space of $L^2$-sections of $\ell$, given by $\left(\rho(\widetilde{g})s\right)(m) = \widetilde{g}^{-1}s(m\cdot\widetilde{g}).$ This representation is unitary with respect to the Liouville measure on $M$. In Section \ref{sec:rep thry ii}, we will see that the quasiregular representation decomposes into a direct sum of unitary irreducible representations $\pi_k:\widetilde{G}\rightarrow End(V_k)$ as (Corollary \ref{cor:L2 decomp})
\begin{equation}
\label{eqn:L2 lk decomp}
L^2(M,\ell^{\otimes k})=4k^2V_k.
\end{equation}
Adapting a general construction due to Richardson \cite{Richardson}, we obtain:

\begin{theorem}
Let $k\in\mathbb{Z}\setminus\{0\}$. For each $j=1,...,4k^2,$ there exists a map\footnote{We will actually construct maps $\Theta_k^j:V_k\rightarrow L_k^2(P)$, where $L_k^2(P)$ is a space of $S^1$-equivariant functions on the circle bundle associated to $\ell$; each such function can be identified with a section of  $\ell^{\otimes k}$.} $\theta_k^j:V_k\rightarrow L^2(M,\ell^{\otimes k})$ such that
\begin{enumerate}
\item $\theta_k^j$ is unitary, up to a constant,
\item $\theta_k^j(V_k)$ is orthogonal to $\theta_k^{j\prime}(V_k)$ whenever $j\neq j^{\prime},$ and
\item $\theta_k^j$ intertwines the actions of $\widetilde{G}$ on
$L^2(M,\ell^{\otimes k})$ and $V_k.$
\end{enumerate}
\end{theorem}

The maps $\theta_k^j$ are generalizations of maps introduced by Weil in \cite{Weil}. In \cite{Brezin}, Brezin considered in detail these maps in the case of Heisenberg groups. In \cite{Brezin}, Brezin also described an inductive procedure to obtain decompositions of the form \eqref{eqn:L2 lk  decomp} for a general nilmanifold, though his procedure is somewhat  different from ours.

Each of the representation spaces $V_k$ is isomorphic to $L^2(H\backslash\widetilde{G})$, where $H$ is any choice of a certain family of subgroups of $\widetilde{G}$: those with Lie algebra subordinate to certain coadjoint orbits, described in Theorem \ref{thm:subordinate  subalgebras}. Both $\widetilde{G}$ and $H$ are nilpotent, hence  exponential, groups. The group $\widetilde{G}$ is diffeomorphic to $\mathbb{R}^5$ while $H\backslash\widetilde{G}$ is diffeomorphic to $\mathbb{R}^2,$ so that $V_k\simeq L^2(\mathbb{R}^2)$. An element of  $L^2(H\backslash\widetilde{G})$ is already constant along $H$-cosets, and we will see in Section \ref{sec:rep thry ii} that $\theta_k^j$ is  essentially a sum over the remaining lattice directions. For this reason,  we call the $\theta_k^j$ \textit{periodizing maps}, even though they  are \textit{not} quite what one usually means by the term (the reason is, again, because we are really dealing with sections of a nontrivializable line bundle rather than functions).

Let $\mathfrak{h}=Lie(H)$ and set $\mathfrak{h}_0=\mathfrak{h}\cap Lie(G)$, where  $G\hookrightarrow\widetilde{G}$ as the zero section. Note that $T_1 G\simeq T_{\Gamma_0}M.$ The next theorem (a concatenation of  Theorem \ref{thm:subordinate subalgebras}, Lemma \ref{lemma:ideal SSAs},  and Theorems \ref{thm:Lag tori} and \ref{thm:sLag}) exposes the symplectic  structure of $M$ in terms of the algebraic structure of $\widetilde{G}$. (We recall the definitions related to Lagrangian subspaces in Section \ref{subsec:Lagr in M}. The notion we use of special Lagrangian is due to Tomassini and Vezzoni \cite{Tomassini-Vezzoni}.)

\begin{theorem}
The left $G$-invariant distribution on $M$ induced by the subspace $\mathfrak{h}_0\subset T_{\Gamma_0}M$ is integrable and Lagrangian,  hence defines a Lagrangian foliation of $M$. Moreover, the set of ideal  subordinate subalgebras can be parameterized by  $e\in\mathbb{R\cup\{\pm\infty\}}$, and the foliation induced by  $\mathfrak{h}^e,~e\in\mathbb{R}$ is special Lagrangian. Finally, the  foliations induced by the subordinate subalgebras $\mathfrak{h}^e,~e=0,\pm\infty$ are Lagrangian torus fibrations.
\end{theorem}

The universal cover of $M$ is $G$ (since $G\simeq\mathbb{R}^4$ is contractible), and so $\ell\rightarrow M$ lifts to a trivializable line bundle $\check{\ell}\rightarrow G$. Upon trivializing $\check{\ell}\simeq G\times\mathbb{C},$ a section $s\in\Gamma(M,\ell)$ yields a function $f_{s}\in G\rightarrow\mathbb{C}$. Such a function is necessarily  pseudoperiodic, that is, it admits \textit{transformation rules} associated  to the lattice elements of the form  $f_{s}(\gamma_0g)=e(g,\gamma_0)f_{s}(g),$ for some multiplier $e(g,\gamma_0)$ which is independent of $f_{s}.$ In particular, given $\phi\in L^2(H\backslash\widetilde{G})$, the periodized image $\theta_k^j\phi\in\Gamma(M,\ell^{\otimes k})$ lifts to a pseudoperiodic  function $\vartheta_k^j\phi:G\rightarrow\mathbb{C}$. In Section \ref{sec:rep thry ii}, we prove the following pseudoperiodicity relations.

\setcounter{section}5 \setcounter{theorem}{6}
\begin{theorem}
Let $\gamma_0\in\Gamma_0.$ Then
\[
(\vartheta_k^j\phi)(\gamma_0g)=\exp\{-4\pi ik\psi(\gamma_0^{-1},g)\}(\vartheta_k^j\phi)(g).
\]
where $\psi(\tilde{g}_1,\tilde{g}_2)$ is defined by the group
multiplication of $\widetilde{G}\simeq G\rtimes\mathbb{R}:$
\[
\tilde{g}_1\cdot\tilde{g}_2=(g_1,x_1)\cdot(g_2,x_2)=(g_1\cdot g_2,x_1+x_2+\psi(g_1,g_2)).
\]
\end{theorem}
\setcounter{section}1 \setcounter{theorem}3

Our final result is a description of the almost K\"{a}hler quantization of $M$, one aspect of which yields a direct proof, in our case, of a general theorem of Guillemin and Uribe \cite{Guillemin-Uribe}. Choose a left-invariant metric on $G$. Associated to the resulting metric on $M$ is a Laplacian $\Delta^{(k)}$ acting on $\Gamma(\ell^{\otimes k}).$ Since it is left-invariant, the Laplacian $\Delta^{(k)}$ induces a Laplacian  $\Delta_k$ acting on $V_k.$

\begin{theorem}
\label{thm:GU}
There exist constants $a,C>0$ such that for $k$ sufficiently large, the  lowest eigenvalue $\lambda_0$ of $\Delta_k-4\pi k$ has multiplicity one (i.e. there is a unique ground state) and is contained in $(-a,a)$.  Moreover, the next largest eigenvalue $\lambda_1$ is bounded below by  $Ck$. The spectrum of $\Delta^{(k)}-4\pi k$ is identical to the spectrum of $\Delta_k-4\pi k$, except that each eigenvalue is repeated with  multiplicity $4k^2.$
\end{theorem}

Throughout the rest of the paper, we will present specific computations exhibiting the above theorems (for specific choices of the relevant structures) in an effort to illustrate the similarities and differences  with the classical theory of $\vartheta$-functions. These computations will  appear under the heading of \textbf{Example}, though they should be  understood as instances of the main results and techniques we discuss.

\paragraph{Example}
In Section \ref{sec:lagr foliations}, we will see that there exists a  subgroup $H^0<\widetilde{G}$ such that the left-invariant Lagrangian  foliation of $M$ induced by $\mathfrak{h}_0=Lie(H^0)\cap Lie(G)$ is a fibration of $M$ by special Lagrangian tori. After choosing a matrix realization of $\widetilde{G}$ (listed in the Appendix) we can identify $G\simeq\mathbb{R}^4$ and $H\backslash\widetilde{G}\simeq\mathbb{R}^2$ (equipped with the Lebesgue measure).

Associated to this data, for each $k\in\mathbb{Z}_{\geq0}$ there is a family of maps
\[
\{\theta_k^{m,n}:L^2(\mathbb{R}^2)\rightarrow L^2(M,\ell^{\otimes k}),~m,n=0,1,\dots,2k-1\}
\]
such that
\[
L^2(M,\ell^{\otimes k})\simeq\bigoplus_{m,n=0}^{2k-1}\theta_k^{m,n}(L^2(\mathbb{R}^2))
\]
is an orthogonal decomposition of $L^2(M,\ell^{\otimes k})$ into irreducible $\widetilde{G}$-spaces (Section \ref{subsec:periodizing maps}).

Identifying sections of $\ell^{\otimes k}$ with sections of the pullback bundle $\check{\ell}^{\otimes k}\rightarrow G\simeq\mathbb{R}^4$ and hence with functions on $\mathbb{R}^4,$ we obtain, in Section \ref{subsec:trans rules}, for each square-integrable function $f:\mathbb{R}^2\rightarrow\mathbb{C}$ and for each $m,n=0,1,\dots,2k-1$ a function $\vartheta_k^{m,n}f:\mathbb{R}^4\rightarrow\mathbb{C}$ given by
\[
(\vartheta_k^{m,n}f)(x,y,z,t)=e^{-2\pi i[my-n(z+xy)]}e^{-4\pi ikzx} \sum_{a,b\in\mathbb{Z}}e^{2\pi inya}e^{-4\pi  ik(by-za-y(x+a)^2/2)}f(x+a,t+b).
\]
These functions satisfy the pseudoperiodicity conditions
\begin{align*}
(\vartheta_k^{m,n}f)(x+1,y,z,t) &=(\vartheta_k^{m,n}f)(x,y,z,t),\\
(\vartheta_k^{m,n}f)(x,y+1,z-x,t) &=e^{-2\pi ikx^2}(\vartheta_k^{m,n}f)(x,y,z,t),\\
(\vartheta_k^{m,n}f)(x,y,z+1,t) &=e^{4\pi ikx}(\vartheta_k^{m,n}f)(x,y,z,t),\text{ and}\\
(\vartheta_k^{m,n}f)(x,y,z,t+1) &=e^{4\pi iky}(\vartheta_k^{m,n}f)(x,y,z,t).
\end{align*}

Moreover, we show in Section \ref{sec:harmonic on P} that if $\psi_0$ is the (unique) ground state of the second-order elliptic differential operator
\[
\Delta_k:=-\partial_{xx}-\partial_{tt}+16k^2\pi^2(x^2+t^2)+16k^2\pi^2x^2\left(\frac{x^2}4-t\right)
\]
then the images $\vartheta_k^{m,n}\psi_0,~m,n=0,1,\dots,2k-1$ are the $\vartheta$-functions associated to the Kodaira--Thurston manifold and form  a basis for the almost K\"{a}hler quantization of $M$.\hfill$\square$

\subsection{\label{subsec:classical thry} The classical theory of $\vartheta$-functions}

We give here a short description of the classical theory of $\vartheta$-functions. Of course, we cannot hope do more than scratch the  surface of this vast subject, so we will content ourselves here with recalling those pieces which suit our present interests (and even these  points will be given a succinct treatment). There are many excellent  references in the literature dealing with $\vartheta$-functions; too many, in fact, for us to give any sort of inclusive list. Nevertheless, we would refer the interested reader to the Tata Lectures of Mumford \cite{Mumford1}, \cite{Mumford2}, for a treatment of $\vartheta$-functions from both the algebraic and geometric point of view; in particular, the point of view taken in the third volume of the series \cite{Mumford3} (Mumford--Nori--Norman) is very much in the same vein as the approach taken  in this paper. For connections with representation theory, and in  particular the deep connections of the theory of $\vartheta$-functions with the theory of nilpotent Lie groups, we recommend the work of Auslander and Tolimieri \cite{Auslander-Tolimieri}. Let us emphasize that the following account of the classical theory of $\vartheta$-functions consists  entirely of well-known material that may be found in the references  mentioned above.

In his \textit{Fundamenta Nova Theoriae Functionum Ellipticarum} \cite{Jacobi}, Jacobi gave the first treatment of what is now known as the $\vartheta$-function, defined as the series
\[
\vartheta(z,\tau)=\sum_{n\in\mathbb{Z}}e^{i\tau\pi n^2+2\pi inz}
\]
where $z\in\mathbb{C}$ and $\tau\in\mathfrak{H}^{+}:=\{z\in\mathbb{C}:\operatorname{Im}z>0\}.$ This series converges absolutely, and uniformly on compact sets. Hence, it defines an entire holomorphic function.

There are no nonconstant periodic entire holomorphic functions, but $\vartheta(z,\tau)$ is, in some sense, as close to periodic as an entire holomorphic function can be; it is easy to verify that
\[
\vartheta(z+1,\tau)=\vartheta(z,\tau)
\]
and, more interestingly,
\[
\vartheta(z+\tau,\tau)=e^{-i\pi\tau-2\pi iz}\vartheta(z,\tau).
\]
Because of these relations, $\vartheta(z,\tau)$ is said to be \textit{pseudoperiodic} with respect to the lattice $\mathbb{Z}+\tau\mathbb{Z}\subset\mathbb{C}.$ If $\vartheta(z,\tau)$ were periodic with respect to the lattice $\mathbb{Z}+\tau\mathbb{Z},$ then it would descend  to a function on the torus $T^2=\mathbb{C}/(\mathbb{Z}+\tau\mathbb{Z}).$  The geometric interpretation of $\vartheta(z,\tau)$ which we will generalize arises from the fact that because of the pseudoperiodicity conditions, $\vartheta(z,\tau)$ descends to a section of a (nontrivializable) line bundle over the torus, rather than a function.

\smallskip
We momentarily shift our point of view and recall some basic symplectic geometry. An action of a Lie group $G$ on a symplectic manifold $(M,\omega)$ is said to be \textit{weakly Hamiltonian} if each $1$-parameter subgroup is infinitesimally generated by the symplectic  gradient of some Hamiltonian function, that is, if for each  $\xi\in\mathfrak{g}:=Lie(G)$ there exists a function  $\phi_{\xi}:M\rightarrow\mathbb{R}$ such that
\[
d\phi_{\xi}=X^{\xi}\lefthook\omega,
\]
where $X^{\xi}$ is infinitesimal action of $\xi$ on $M.$ Such an action is \textit{Hamiltonian} if the linear map $\xi\rightarrow\phi_{\xi}$ is a Poisson--Lie homomorphism, that is, if
\[
\{\phi_{\xi},\phi_{\eta}\}=\phi_{[\xi,\eta]}.
\]

Consider $\mathbb{R}^2$ with coordinates $(x,y)$ equipped with the standard symplectic form $\omega=dx\wedge dy.$ The Abelian group $\mathbb{R}^2$ acts on itself by translations which are infinitesimally generated by the vector fields $\partial_{x}$ and $\partial_{y}$. Moreover, this action is weakly Hamiltonian; indeed $\partial_{x}$ is the Hamiltonian  flow of the function $\phi_{x}:=y$ and $\partial_{y}$ is the Hamiltonian flow of $\phi_{y}:=x.$ A quick calculation, though, shows that $\{\phi_{x},\phi_{y}\}=1$, whereas $[\partial_{x},\partial_{y}]=0$ implies  $\phi_{[\partial_{x},\partial_{y}]}=0.$ Hence, the action of $\mathbb{R}^2$ on itself by translations is \emph{not} Hamiltonian.

Let us reflect on this situation for minute. On the one hand, $[\partial_{x},\partial_{y}]=0$ defines the Lie algebra structure of $\mathbb{R}^2.$ On the other hand, we would like a Lie algebra structure which is reflected as a Poisson algebra satisfying $\{\partial_{x},\partial_{y}\}=1$ (if we want a Hamiltonian action, that is). The resolution, it seems, is to take a central extension of $\mathbb{R}^2$ whose Lie algebra structure is given by $[\partial_{x},\partial_{y}]=Z$ and assign the Hamiltonian function $\phi_{Z}:=1.$ This means that $Z$ acts trivially on $\mathbb{R}^2$, but the new group acts in a Hamiltonian fashion. This new group is, of course, the well-known Heisenberg group described by the short exact sequence
\[
0\rightarrow\mathbb{R} \rightarrow \operatorname{Heis}(3) \rightarrow \mathbb{R}^2\rightarrow0.
\]

The Heisenberg group can be realized in many equivalent ways. For what comes later, we will find it convenient to make the definition
\[
\operatorname{Heis}(3):=\left\{\mathbf{a}\in\mathbb{R}^3\right\}
\]
equipped with the group law
\[
\mathbf{a}\cdot\mathbf{b}=(a^1+b^1,a^2+b^2,a^3+b^3-a^2b^1).
\]
Note that the first two components give the action of $\mathbb{R}^2$ on itself by translations as claimed. That the Lie algebra of this Lie group satisfies the bracket relations $[X,Y]=Z,$ with $X=\partial_{a^1}^{R},\  Y=\partial_{a^2}^{R},$ and $Z=\partial_{a^3}^{R}$, is an exercise left for the reader\footnote{We use \emph{right} invariant vector fields because they are the generators of the \emph{left} action by translations.}.

Let $\Gamma:=\{\mathbf{a}\in\operatorname{Heis}(3):\mathbf{a}\in\mathbb{Z}^3\}$ denote the integer lattice in the  Heisenberg group. The quotient
\[
Q:=\Gamma\backslash\operatorname{Heis}(3)
\]
is a compact manifold. In fact, the center $\{(0,0,z)\}\subset\operatorname{Heis}(3) $ of the Heisenberg group acts  (on the right) as $S^1$ on $Q$, and this action gives $Q$ the structure of a principal $S^1$-bundle over the torus $T^2$ whose Chern class is the class of the symplectic form (appropriately normalized).

The circle $S^1$ acts on $\mathbb{C}$ by multiplication, and this action induces a Hermitian line bundle $\ell\rightarrow T^2$ associated to $Q$. It turns out that this bundle has a unique (up to normalization)  holomorphic section. Pulling back $\ell$ by the quotient map  $\mathbb{R}^2\rightarrow T^2$ yields a trivializable line bundle over  $\mathbb{R}^2$, and up to factors arising from the choice of  trivialization, $\vartheta(z)$ is this unique holomorphic section represented as a section of the pullback bundle. We will see this much more  explicitly in a moment.

We can view this appearance of $\vartheta(z)$ as a section of a line  bundle over $T^2$ through the lens of the representation theory of the  Heisenberg group (the utility of this approach is that we can generalize it to the Kodaira--Thurston manifold).

The Heisenberg group acts on $Q$ transitively on the right, and this action induces a unitary action on  $L^2(Q)$ (with respect to the Lebesgue measure) via
\[
[\rho(g)f](x)=f(x\cdot g)
\]
which is known as the right (quasi-)regular representation. Thus, $L^2(Q)$ can be decomposed into unitary irreducible representations of $\operatorname{Heis}(3).$

The unitary irreducible representations of $\operatorname{Heis}(3)$ are well-known (and easily computed, see for example \cite[Sec. 2.3]{Kirillov}). For our purposes, it is sufficient to know that for each  $\lambda\in\mathbb{R}\setminus\{0\}$ there exists a unitary irreducible representation $\pi_{\lambda}$ of $\operatorname{Heis}(3)$ on $V_{\lambda}\simeq L^2(\mathbb{R},dx)$ given by
\[
[\pi_{\lambda}((a,b,c))f](x)=e^{2\pi i\lambda(c+bx)}f(x+a).
\]

The decomposition of $L^2(Q)$ into unitary irreducible representations of $\operatorname{Heis}(3)$ is then
\begin{equation}
L^2(Q)=\bigoplus_{k\in\mathbb{Z}}|k|V_k\oplus V_0\label{eqn:L2-decomp1}
\end{equation}
where $V_0\simeq L^2(T^2)$ and each invariant subspace $|k|V_k$ can be decomposed into $|k|$ copies
of the irreducible space $V_k$ (this result seems to be folklore in representation theory, but we refer the reader to \cite[Theorem 1]{Auslander-Brezin} for one proof).

Indeed, a very fruitful (at least in this paper) question to ask is: how is the decomposition \eqref{eqn:L2-decomp1} achieved? That is, given a function $f\in L^2(\mathbb{R},dx),$ how does one obtain a function in $|k|V_k\subset L^2(P)$, and moreover, is there some systematic way to achieve the decomposition of an invariant subspace of $L^2(P)$ into $|k|$ orthogonal copies of $V_k$?

The answer to both of these questions is achieved by the Weil--Brezin $\Theta$-map \cite{Brezin}, \cite{Weil}. Let $x,y$ and $\phi$ be coordinates on $Q$ induced by the coordinates $a^1,\ a^2$ and $a^3$ on $\operatorname{Heis}(3)$. For each $k\in\mathbb{Z}\setminus\{0\}$, define a map $\Theta_k:L^2(\mathbb{R},dx)\rightarrow L^2(Q)$ by
\[
(\Theta_kf)(\Gamma_0(x,y,\phi))=e^{-2\pi ik\phi}\sum_{m\in\mathbb{Z}}f(y+m)e^{2\pi imx}.
\]
Each $\Theta_k$ is unitary and intertwines the action of the Heisenberg group. Define a function\footnote{Each function $\Theta_k f$ can be identified with a section of $\ell^{\otimes k}\rightarrow T^2.$ The universal cover of  $T^2$ is $\mathbb{R}^2,$ and so $\ell$ can be lifted to a trivializable line bundle $\check{\ell}\rightarrow\mathbb{R}^2$. So the section associated to $\Theta_k f$ induces a section of the lifted bundle $\check{\ell}^{\otimes  k}\rightarrow\mathbb{R}^2$. After choosing a certain trivialization of  $\check{\ell}$, and hence identifying sections with functions, one obtains $\vartheta_k f$.} $\vartheta_k f:\mathbb{R}^2\rightarrow\mathbb{C}$ by
\[
(\vartheta_k f)(x,y)=\sum_{m\in\mathbb{Z}}f(y+m)e^{2\pi imx}.
\]
The function $\vartheta_k f$ is square-integrable on any fundamental domain of  $T^2=\mathbb{R}^2/\mathbb{Z}^2,$ though not on $\mathbb{R}^2$ itself.

Since it will make no difference for our purposes, we will henceforth take $\tau=i$ and write $\vartheta(z):=\vartheta(z,i).$

To see how $\vartheta(z)$ arises from the Weil--Brezin map requires one more piece of the puzzle. The basic fact of the matter is that $\vartheta(z)$ is, up to exponential factors, the image of the standard Gaussian under the Weil--Brezin map with $k=1$, after $\ell\rightarrow T^2$ has been lifted to a trivial line bundle $\check{\ell}\simeq\mathbb{R}^2\times\mathbb{C}:$
\begin{equation}
[\vartheta_1(e^{-\pi t^2})](x,y,\phi)=\vartheta(x+iy)\times e^{-\pi y^2}.\label{eqn:periodized theta}
\end{equation}
Equation \eqref{eqn:periodized theta} is not as \textit{ad hoc} as it seems at first sight: the factor $e^{-\pi y^2}$ arises from the choice of trivialization of $\check{\ell}$. Why $\vartheta$ arises as the image of the standard Gaussian requires a bit more explanation.

Sections of $\ell^{\otimes k}$ can be identified with functions $f:Q\rightarrow\mathbb{C}$ which satisfy the $S^1$-equivariance condition
\begin{equation}
\label{eqn:theta S1-equiv}
f((0,0,c)\cdot(x,y,\phi))=e^{-2\pi ikc}f((x,y,\phi)).
\end{equation}

Consider the Hodge Laplacian\footnote{We take the connection on $\ell$ induced by the connection on $P$ whose connection $1$-form is dual to $Z$ via the left-invariant metric which makes $\{X,Y,Z\}$ an orthonormal basis.} acting on sections of $\ell.$ It induces a second order elliptic differential operator $\Delta^{(1)}$ acting on $L_1^2(P)$ which can be written in terms of the right quasi-regular representation as
\[
\Delta^{(1)}=-\frac14\left[  \rho_{\ast}(X)^2+\rho_{\ast}(Y)^2+2\pi\right].
\]

Since a function which transforms according to \eqref{eqn:theta S1-equiv} satisfies the same $S^1$-equivariance as $V_1,$ the Weil--Brezin map $\Theta_1$ restricts to an $S^1$-equivariant map $\Theta_1:V_1\rightarrow L_1^2(P).$ The Hodge Laplacian $\Delta^{(1)}$  then yields a differential operator $\Delta_1$ acting on $V_1$ which is given by
\begin{align*}
\Delta_1 &= -\frac14\left[(\pi_1)_{\ast}^2(X)+(\pi_1)_{\ast}^2(Y)+2\pi\right] \\
&=-\frac14\left[\partial_t^2-4\pi^2t^2+2\pi\right].
\end{align*}

On $\ell\rightarrow T^2,$ the kernel of the Hodge Laplacian consists exactly of holomorphic sections (Hodge's theorem). On the other side, the kernel of $\Delta_1,$ acting on $V_1\simeq L^2(\mathbb{R}),$ is spanned by the Gaussian $e^{-\pi t^2}$. Hence, we see that \eqref{eqn:periodized theta} is simply an expression of the kernel of the Hodge Laplacian acting on $S^1$-equivariant functions on $P$ from two different points of view.

\subsection{Quantization}

Classical $\vartheta$-functions, and also the $\vartheta$-functions constructed in this paper, are related to a construction in mathematical physics known as geometric quantization. We will not go into detail about geometric quantization (the interested reader may refer to \cite{Woodhouse} for comprehensive account), but since it will eventually provide the structure which we  generalize, we now describe the relevant pieces.

Geometric quantization provides a systematic recipe which associates to each symplectic manifold $(M,\omega)$ a Hilbert space $\mathcal{H}_{M}$ and a map $Q$ from (some subalgebra of) $C^{\infty}(M)$ to the set of (usually unbounded) operators on $\mathcal{H}_{M}.$ This association is rigged in such a way as to  be nontrivial and approximately functorial. The construction works best when $M$ is actually K\"{a}hler (for example, on the torus).

Suppose $M$ is a compact K\"{a}hler manifold with integral symplectic form $\omega.$ Then there is a Hermitian line bundle $\ell\rightarrow M$ with compatible connection with first Chern class the class of $\omega,$ called a \textit{prequantum line bundle}. In this situation, for each $k\in \mathbb{Z}_{+}$ one defines the \textit{quantum Hilbert space} to be the $L^2$-space of holomorphic sections of $\ell^{\otimes k}$:
\[
\mathcal{H}_{M}:=H^0L^2(M,\ell^{\otimes k}).
\]

By Hodge's theorem, the quantum Hilbert space is precisely the kernel of the Hodge Laplacian $\Delta_k$ acting on $\ell^{\otimes k}.$ Hence, we see that the \textbf{geometric quantization of the torus consists exactly of }$\vartheta$\textbf{-functions}.

In order to generalize the construction of $\vartheta$-functions, one should study the geometric quantization of other manifolds\footnote{In fact, this idea  leads to many interesting examples for different choices of K\"{a}hler manifold  $M.$ For example, if $M$ is taken to be a Riemann surface of genus $g\geq2,$  then geometric quantization yields modular forms.}. The description we have  given, though, makes critical use of the assumption that $M$ is K\"{a}hler  (otherwise, we either have no notion of ``holomorphic'', if $M$ is not complex, or there might be no holomorphic sections, if $M$ is complex but the line bundle $\ell$ is not positive).

We will consider one possible generalization of the basic program of geometric quantization, known as \emph{almost K\"{a}hler quantization}. Although the general theory has been around for some time, no nonK\"{a}hler examples of this method have been worked out. (It was part of the original motivation of the current work to produce such an example.)

Suppose that $(M^{2n},\omega)$ is a compact symplectic manifold and that the class $\left[\omega/2\pi\right]$ is integral, whence there exists a prequantum line bundle $\ell\rightarrow M$. Choose a metric $g$ on $M$, and construct the rescaled metric Laplacian $\Delta^{(k)}-2\pi nk$ acting on sections of $\ell^{\otimes k}.$ Denote the spectrum of $\Delta^{(k)}-2\pi nk$ by
\[
\lambda_1\leq\lambda_2\leq\cdots\leq\lambda_{n}\leq\cdots\rightarrow\infty.
\]

If $M$ is K\"{a}hler, and $g$ is the K\"{a}hler metric, then $\Delta^{(k)}-2\pi nk$ is equal to the Hodge Laplacian. So the geometric quantization of $M$ consists of the kernel of $\Delta^{(k)}-2\pi nk.$ In the nonK\"{a}hler case, even though there is no Hodge Laplacian, it still makes sense to study $\Delta^{(k)}-2\pi nk.$ The difficulty is that its kernel is generically empty.

The basic foundation on which almost K\"{a}hler quantization rests is that there is an approximate kernel of the rescaled metric Laplacian, described by the following theorem of Guillemin and Uribe \cite{Guillemin-Uribe}. Let the eigenvalues of $\Delta^{(k)}-2\pi nk$ be denoted by $\lambda_0\leq\lambda_1\leq\lambda_2\leq\cdots\rightarrow\infty$.

\begin{theorem}
There exist constants $a\geq0$ and $C>0$ such that for $k$ sufficiently large,
\begin{enumerate}
\item $\lambda_{j}\in(-a,a)$ for $j=1,2,\dots,d_k,$ and
\item $\lambda_{d_k+1}>C k,$
\end{enumerate}
where $d_k:=RR(M;\ell^{\otimes k})$ is the Riemann--Roch number of $M$ twisted by $\ell^{\otimes k}.$
\end{theorem}

The important point is that the constants $a$ and $C$ are independent of $k.$ Thus, the span of the eigenfunctions of $\Delta^{(k)}-2\pi nk$ with bound eigenvalues constitutes an approximate kernel. Indeed, if $M$ is K\"{a}hler, then these bound eigenvalues are all exactly zero, and the span of the corresponding eigenfunctions is the kernel of the Hodge Laplacian.

Following this line of reasoning, Borthwick and Uribe \cite{Borthwick-Uribe} defined the almost K\"{a}hler quantization of $(M,\omega)$ to be the span of the bound eigenfunctions of the rescaled metric Laplacian:
\[
\mathcal{H}_{M}:=\operatorname{span}_{\mathbb{C}}\{\psi\in\Gamma(\ell^{\otimes k}):(\Delta^{(k)}-2\pi nk)\psi=\lambda\psi\text{ with }\lambda\in(-a,a)\}.
\]
These bound eigenfunctions, in the case of the Kodaira--Thurston manifold, are the desired generalization of $\vartheta$-functions.

\subsection{Summary}

In Section \ref{sec:preliminaries}, we describe a nontrivial central extension $1 \rightarrow \mathbb{R} \rightarrow \widetilde{G} \rightarrow G:=\operatorname{Heis}(3)\times \mathbb{R} \rightarrow 1$ which plays a central role in our analysis. The quotient of $\widetilde{G}$ by an integer lattice yields a principal circle bundle $P$ over the Kodaira--Thurston manifold. Section \ref{sec:preliminaries} contains a description of the geometry of $P$, the complex line bundles  $\ell^{\otimes k},~k=1,2,\dots$ associated to $P,$ and their lifts to (trivializable) line bundles over $G$, which are the source of $\vartheta$-functions associated to $M$. Section \ref{sec:preliminaries} concludes with a review of the symplectic geometry we use later in the paper (Lagrangian and special Lagrangian foliations and fibrations).

Section \ref{sec:rep thry i} begins an analysis of the representation theory of $\widetilde{G}$; we use Kirillov's orbit method to construct the unitary irreducible representations of $\widetilde{G}$. After a brief review of the induction procedure (the basis of the orbit method), we discuss the set of subordinate subalgebras (the choice of which is the first step in orbit  method).

The subordinate subalgebras provide the link with the symplectic geometry of $M$. In Section \ref{sec:lagr foliations}, we describe the correspondence between subordinate subalgebras and Lagrangian foliations. We then show that a certain subfamily of subordinate subalgebras, consisting exactly of the ideal subordinate subalgebras, correspond to special Lagrangian foliations. We also describe Lagrangian torus fibrations of $M$.

In Section \ref{sec:rep thry ii}, we return to the representation theory of $\widetilde{G}$. In this section, we find a decomposition of $L^2(P)$ into unitary irreducible representations of $\widetilde{G}$. We also describe the periodizing maps $\Theta_k^j$ which realize this decomposition, and discuss the pseudoperiodicty of the images of $\Theta_k^j.$

In the final Section \ref{sec:harmonic on P} we consider the harmonic analysis of $P$. After discussing the various Laplacians in the picture, we use semiclassical methods (in particular, the quantum Birkhoff canonical form) to analyze their spectra. Finally, we are able to define the $\vartheta$-functions associated to $M$ and hence the almost K\"{a}hler quantization of $M$.

\section{Preliminaries\label{sec:preliminaries}}

We begin by constructing a symplectic nonK\"{a}hler $4$-manifold $(M,\omega)$, known as the Kodaira--Thurston manifold. It is the product of $S^1$ and the quotient of the $3$-dimensional Heisenberg group by a discrete uniform subgroup (that is, a discrete subgroup such that the quotient is compact). We will normalize $\omega$ so that $\left[\omega/2\pi\right]$ is an integral cohomology class.

Let $G=\operatorname{Heis}(3)\times\mathbb{R}$ be the product of the three dimensional Heisenberg group with $\mathbb{R}.$ Convenient faithful matrix representations of this group, as well as those defined below, are given in the Appendix. We will write $\mathbf{g}\in G$ as
\[
\mathbf{g}=(\mathbf{a},r):=(a^1,a^2,a^3,r),~\mathbf{a}\in \operatorname{Heis}(3),~r\in\mathbb{R}.
\]
The group law is given by
\begin{equation}
\label{eqn:G group law}
\left(\mathbf{a},r\right)\cdot\left(\mathbf{b},s\right) = (\mathbf{a}\cdot\mathbf{b},r+s) = \left(a^1+b^1,a^2+b^2,a^3+b^3-a^2 b^1,r+s\right).
\end{equation}

Fix a basis $\{X_1,X_2,X_3,T\}$ of $\mathfrak{g}=Lie(G)=Lie(\operatorname{Heis}(3))\oplus\mathbb{R}$ which satisfies the usual commutation relation $[X_1,X_2]=X_3.$ The coordinates $(\mathbf{a},r)$ on $G$ may be expressed in terms of this basis:
\[
(\mathbf{a},r)=\exp(a^1X_1)\exp(a^2X_2)\exp(a^3X_3)\exp(rT).
\]
Such coordinates on $G$ are called \textit{canonical coordinates}.

Let $\Gamma_0\subset G$ be the integral lattice
\[
\Gamma_0=\{(\mathbf{a},r):a^j,r\in\mathbb{Z}\}.
\]
It is easy to check that $\Gamma_0$ is a subgroup (not normal) of $G.$ The \textit{Kodaira--Thurston manifold} is
\[
M:=\Gamma_0\backslash G.
\]
It is compact and symplectic, as we will see below, but not K\"{a}hler \cite{Thurston}.

A left invariant coframe on $G$ is
\begin{equation}
\label{eqn:glinvcoframe}
\beta_{L}^1=da^1,\ \beta_{L}^2=da^2,\ \beta_{L}^3=da^3+a^2 da^1,\ \beta_{L}^T=dr.
\end{equation}
There is a right invariant frame which corresponds to the dual of the above frame at the identity:
\begin{equation}
\label{eqn:grinvframe}
X_1^{R}=\frac{\partial}{\partial a^1},\ X_2^{R}=\frac{\partial}{\partial a^2}-a^1\frac{\partial}{\partial a^3},\ X_3^{R}=\frac{\partial}{\partial a^3},\ T^{R}=\frac{\partial}{\partial r}.
\end{equation}
Recall that it is the \emph{right} invariant frame that generates the left action of $G$ on itself.

A left invariant symplectic form, normalized so that $\int_{X}[\frac{\omega}{2\pi}]=1$, is
\[
\omega=2\pi\left(\beta_{L}^2\wedge\beta_{L}^3+\beta_{L}^2\wedge\beta_{L}^T\right) = 2\pi\,(da^1\wedge da^3+da^2\wedge dr).
\]
For easy reference and to fix sign conventions, recall that the Hamiltonian vector field $X_{f}$ associated to $f\in C^{\infty}(G)$ is given by
\begin{align*}
X_{f}\lefthook\omega &:= \omega(X_{f},\cdot)=df \\
\Rightarrow X_{f} &= \frac1{2\pi}\left(\frac{\partial f}{\partial a^1}\frac{\partial}{\partial a^3}-\frac{\partial f}{\partial a^3}\frac{\partial}{\partial a^1}+\frac{\partial f}{\partial a^2}\frac{\partial}{\partial r}-\frac{\partial f}{\partial r}\frac{\partial}{\partial a^2}\right).
\end{align*}
The corresponding Poisson brackets are $\{f,g\}=X_{f}(g)=\omega(X_{f},X_{g}).$

\begin{lemma}
The right invariant vector fields $X_1^{R},X_2^{R},X_3^{R},T^{R}$ are Hamiltonian with respect to $\omega.$ A choice of Hamiltonians is
\[
\phi_1=-2\pi a^3,\ \phi_2=-2\pi\left(r+\frac{(a^1)^2}2\right),\ \phi_3=2\pi a^1,\ \phi_T=2\pi a^2.
\]
The induced linear map $\mathfrak{g}\rightarrow C^{\infty}(G)$ is, however, not a Lie algebra homomorphism.
\end{lemma}

\begin{proof}
The first part is an easy computation. For the rest, observe that
\begin{equation}
\label{eqn:pbrelations}
\{\phi_1,\phi_2\}=-\phi_3,\ \{\phi_1,\phi_3 \}=2\pi=\{\phi_2,\phi_T\}.
\end{equation}
\end{proof}

\smallskip
With the conventions thus far, writing $\mathbf{0}=(0,0,0,0)\in G,$ we have
\begin{equation}
\label{eqn:vfrelations}
\begin{split}
&  [X_1^{L},X_2^{L}]_{\mathbf{0}}=X_3=[X_1,X_2]\\
&  [X_1^{R},X_2^{R}]_{\mathbf{0}}=-X_3=-[X_1,X_2].
\end{split}
\end{equation}
The fact that the Poisson brackets above do not close in $\mathfrak{g}$ is an analogy of what happens in the case of translations on $\mathbb{R}^2$ (see Section \ref{subsec:classical thry}). Therefore, we are lead to consider the analogue of the Heisenberg group associated with $G$; namely, a specific central extension $\tilde{\mathfrak{g}}=\operatorname{span}_{\mathbb{R}}\{X_1,X_2,X_3,T,U\}$ of $\mathfrak{g}$ subject to the relations
\begin{equation}
\label{eqn:gt alg relations}
[X_1,X_2]=X_3,\ [X_1,X_3]=-U=[X_2,T].
\end{equation}
The central extension $\tilde{\mathfrak{g}}$ is a three step nilpotent algebra whose center is spanned by $U.$

Observe that we are \emph{not} using the bracket relations \eqref{eqn:pbrelations}. This is because we want the restriction of the Lie algebra of $\widetilde{G}=\exp(\tilde{\mathfrak{g}})$ to the standard embedded copy of $G$ to coincide with the algebra of \emph{left} invariant vector fields  along that embedded copy of $G.$ Hence, due to \eqref{eqn:vfrelations}, we need the change of signs.

\subsection{\label{subsec:prequantum bundles} Prequantum bundles}

The Lie group $\widetilde{G}$ with Lie algebra $\widetilde{\mathfrak{g}} =\operatorname*{span}_{\mathbb{R}}\{X_1,X_2,X_3,T,U\}$ subject to the relations \eqref{eqn:gt alg relations} has the structure of a central extension
\[
0 \rightarrow \mathbb{R} \rightarrow \widetilde{G} \rightarrow G \rightarrow \mathbf{0},
\]
where $\widetilde{G}=G\rtimes\mathbb{R}$. The group product is
\[
(\mathbf{g},v)\cdot(\mathbf{g}^{\prime},w)=(\mathbf{g}\cdot\mathbf{g}^{\prime}, v+w+\psi(\mathbf{g},\mathbf{g}^{\prime}))
\]
where $\psi:G\times G\rightarrow\mathbb{R}$ is given in \eqref{eqn: grp law scalar} below.

An element $(\mathbf{g},v)$ of $\widetilde{G}$ can be written in canonical coordinates as
\[
(\mathbf{g},v)=\exp(a^1X_1)\exp(a^2X_2)\exp(a^3X_3)\exp(rT)\exp(vU).
\]
The group law in these coordinates, which can be worked out either with the Baker--Campbell--Hausdorff formula or the faithful matrix representation given in the Appendix, is
\begin{align}
(\mathbf{a},r,v)\cdot(\mathbf{b},s,w) &= ((\mathbf{a},r)\cdot(\mathbf{b},s), v+w+a^3b^1-{\tfrac12}a^2(b^1)^2+rb^2)\label{eqn:gt grp law}\\
&= (a^1+b^1,a^2+b^2,a^3+b^3-a^2b^1,r+s, v+w+a^3b^1-{\tfrac12}a^2(b^1)^2+rb^2).\nonumber
\end{align}
In particular,
\begin{equation}
\label{eqn: grp law scalar}
\psi((\mathbf{a},r),(\mathbf{b},s))=a^3b^1-\frac12 a^2(b^1)^2+rb^2.
\end{equation}

A left $\widetilde{G}$-invariant frame which corresponds to $\{X_1,X_2,X_3,T,U\}$ at the origin is given by
\[
\begin{split}
X_1^{L}=\frac{\partial}{\partial a^1}-a^2\frac{\partial}{\partial a^3} & +a^3\frac{\partial}{\partial v},\quad X_2^{L}=\frac{\partial}{\partial a^2} +r\frac{\partial}{\partial v},\quad X_3^{L} =\frac{\partial}{\partial a^3},\\
& X_T^{L}=\frac{\partial}{\partial r},\quad X_U^{L}=\frac{\partial}{\partial v}.
\end{split}
\]
The dual left $\widetilde{G}$-invariant coframe is
\begin{equation}
\label{eqn:gtlinvcoframe}
\begin{split}
\beta_{L}^1 =da^1,\quad\beta_{L}^2 &= da^2,\quad\beta_{L}^3 = da^3+a^2da^1,\quad\beta_{L}^T = dr,\\
\beta_{L}^U &=dv-a^3da^1-rda^2.
\end{split}
\end{equation}
Throughout the paper, if we need to choose a metric (for example in Section \ref{sec:harmonic on P}), we will use the left $\widetilde{G}$-invariant metric
\[
g=\left(\beta_{L}^1\right)^2 +\left(\beta_{L}^2\right)^2 +\left(\beta_{L}^3\right)^2 +(\beta_{L}^T)^2 +(\beta_{L}^U)^2.
\]

At the origin in $\widetilde{G},$ this metric yields a symmetric bilinear quadratic form, and orthogonal projection from $\widetilde{\mathfrak{g}}$ to $\mathfrak{g}$ with respect to this form is given by (with the summation convention)
\begin{equation}
\label{eqn:g-proj gt->g}
x^jX_{j}+tT+uU\mapsto x^jX_{j}+tT.
\end{equation}
Moreover, the restriction of the metric $g$ to $G$ yields a metric (which we denote by the same symbol)
\[
g=\left(\beta_{L}^1\right)^2 +\left(\beta_{L}^2\right)^2 +\left(\beta_{L}^3\right)^2 +\left(\beta_{L}^T\right)^2.
\]

\smallskip
The fundamental importance of $\widetilde{G}$ to our analysis is due to the  following.

\begin{lemma}
The group $\widetilde{G}$ acts on $(G,\omega)$ in a Hamiltonian fashion, provided we associate to $U$ the Hamiltonian $\phi_U=2\pi.$
\end{lemma}

The center of $\widetilde{G}$ is $\exp(\mathbb{R}\cdot U)$ and can be identified with $\mathbb{R}$ if we set $\exp(U)\mapsto1$. Denote by $Z\subset\widetilde{G}$ the subgroup of the center corresponding to the half-integers\footnote{We are forced to consider half-integers because of the $\frac12$ that appears in $\psi$ \eqref{eqn: grp law scalar}.}, that is, $Z=\{(\mathbf{g},\frac{n}2):n\in\mathbb{Z}\}$. Then $K=\widetilde{G}/Z$ is a group with center $S^1\simeq\mathbb{R}/\frac12\mathbb{Z}.$ Indeed, $K$ is an $S^1$-central extension of $G$
\[
[0]\rightarrow S^1\rightarrow K\rightarrow G\rightarrow\mathbf{0}.
\]
$K$ is the group of elements $(\mathbf{a},r,[v]),$ where $[v]$ is the class of $v$ modulo $\frac12\mathbb{Z}$.

The group homomorphism $F:\widetilde{G}\rightarrow G$ covering the Lie algebra homomorphism $U\mapsto0$ is
\[
F((\mathbf{a},r,v))=(\mathbf{a},r).
\]
The homomorphism $F$ induces a homomorphism from $K$ to $G$ which we continue to denote by $F:$
\[
F((\mathbf{a},r,[v]))=(\mathbf{a},r).
\]

Let us denote by $\Gamma_k=\{(\gamma_0,[0])\in K:\gamma_0\in\Gamma_0\}.$ Then $F(\Gamma_k)=\Gamma_0,$ $\Gamma_k$ is  a lattice in $K$, and $F$ induces a map
\[
\pi:P:=\Gamma_k\backslash K\rightarrow M=\Gamma_0\backslash G.
\]
The projection $\pi$ and the $S^1$-action given by right multiplication by the center of $K$, i.e.,
\[
\Gamma_k(\mathbf{g},[v])\cdot e^{2\pi i\theta}:=\Gamma_k(\mathbf{g},[v+\theta/2]),
\]
give $P$ the structure of a principal $S^1$ bundle.

Equivalently, we can define an integral lattice in $\widetilde{G}$
\[
\widetilde{\Gamma}=\{(\mathbf{\gamma}_0,v)\in\widetilde{G} : \gamma_0\in\Gamma_0,v\in\frac12\mathbb{Z}\}
\]
and then identify $P=\Gamma_k\backslash K=\widetilde{\Gamma}\backslash\widetilde{G}.$

\begin{lemma}
\label{lemma:preQ circle bundle}
$P$ is a prequantum circle bundle over $X$, that is, a circle bundle with connection whose curvature\footnote{From the geometric point of view, it would be more natural to define $U^{\prime}=-\sqrt{-1}U$ and then identify the center of $K$ with $S^1$ via  $\exp(2\pi\sqrt{-1}U^{\prime})\mapsto1$. The fiber of the $S^1$-bundle $P$ would then have tangent space $2\pi\sqrt{-1}\mathbb{R}$. But since  $\widetilde{G}$ (and hence $K$) is a real Lie group, we omit the algebraically  wieldy factors of $2\pi\sqrt{-1}$. This is responsible for the fact that $P$  has curvature $\omega$ instead of the more standard $-\sqrt{-1}\omega$.} is $\omega$.
\end{lemma}

\begin{proof}
By \eqref{eqn:glinvcoframe}, we have
\[
d\beta_{L}^U=da^1\wedge da^3+da^2\wedge dr.
\]
The right hand side above is exactly $1/2\pi$ times the pullback to $\tilde{\mathfrak{g}}$ of the symplectic form on $\mathfrak{g}.$ Hence we can take $2\pi\beta^U$ for a connection $1$-form. This means $\pi:P\rightarrow M$ is indeed a prequantum circle bundle.
\end{proof}

\smallskip
Since the universal cover of $M$ is $G$, the circle bundle $P$ lifts to a circle bundle over $G$, and this circle bundle is just $K$.

We define the prequantum line bundle $\ell\rightarrow M$ to be the Hermitian line bundle associated to $P$ equipped with the connection induced by the connection $1$-form $2\pi\beta^U$. Recall that for a principal $G$-bundle $P\rightarrow M,$ if $\rho:G\rightarrow End(E)$ is a representation of $G$, then the vector bundle associated to $P$ with fiber $E$ is defined by
\[
P\times_{\rho}E:=\{[(p,v)]\},
\]
where the equivalence is given by $(p,v)\sim(p\cdot g,\rho(g^{-1})v).$

Let
\begin{equation}
\label{eqn:line bundle S1 action}
\rho^{(k)}(e^{2\pi i\theta})=e^{4\pi ik\theta}.
\end{equation}
Observe that this is \textit{not} the standard action of $S^1$ on $\mathbb{C}$; we have introduced an extra factor of $2$ to compensate for the fact that the center of $K$ is isomorphic to $\mathbb{R}/\frac12\mathbb{Z}$. The line bundles associated to $P$ by this action are, for $k\in\mathbb{Z}_{>0}$,
\[
\ell^{\otimes k}=P\times_{\rho^{(k)}}\mathbb{C}.
\]
The line bundle $\ell^{\otimes k}$ is equipped with a covariant derivative induced by the connection $1$-form $2\pi\beta^U.$ The curvature of this connection is therefore $4\pi k\omega$ and so the Chern class of $\ell^{\otimes k}$ is $[2k\omega]$; again, the factor of $2$ arises because of the $\frac12$ that appears in \eqref{eqn: grp law scalar}.

The lattice $\Gamma_0$ acts on $K\times_{\rho^{(k)}}\mathbb{C}$ by
\[
\gamma_0\cdot[((\mathbf{g},[v]),z)] =[((\gamma_0,[0])\cdot(\mathbf{g},[v]),z)].
\]
Hence, there is a canonical isomorphism of line bundles
\[
(\Gamma_k\backslash K)\times_{\rho^{(k)}}\mathbb{C} \simeq \Gamma_0\backslash(K\times_{\rho^{(k)}}\mathbb{C})\rightarrow\Gamma_0\backslash G=M.
\]
The lift of $\ell^{\otimes k}$ to $G$ is therefore the line bundle $\check{\ell}^{\otimes k}\rightarrow G$ associated to $K$:
\[
\check{\ell}^{\otimes k}:=K\times_{\rho^{(k)}}\mathbb{C}.
\]

The computations in this paper are greatly simplified by identifying sections of the prequantum line bundle $\ell$ (resp. $\check{\ell}$) with $S^1$-equivariant functions on the total space of the associated prequantum circle bundle $P$ (resp. $K$). The following lemma is standard, see for  example \cite[Prop. 1.7]{BGV}.

\begin{lemma}
\label{lemma:identification}
Let $L_k^2(P)$ denote the space of square-integrable $\mathbb{C}$-valued maps on $P$ which satisfy the equivariance $f(pe^{2\pi i\theta})=e^{-4\pi ik\theta}f(p)$. There is a natural isomorphism between $L_k^2(P)$ and the space $L^2(M,\ell^{\otimes k})$ given by associating $\tilde{s}\in L_k^2(P)$ to the section $s\in L^2(M,\ell^{\otimes k})$ defined by $s(x)=[(p,\tilde{s}(p))] $ where $p$ is any point in $P_{x}$ (i.e. $\pi(p)=x$).
\end{lemma}

\subsection{Lagrangians in $M\label{subsec:Lagr in M}$}

Let $(M^4,\omega)$ be a compact symplectic $4$-manifold. A subspace $L_{m}\subset T_{m}M$ is said to be \textit{Lagrangian} if $\dim L_{m}=2$ and $\left.\omega_{m}\right\vert_{L}=0$. A submanifold $N\hookrightarrow M$ is Lagrangian if $T_{m}N\subset T_{m}M$ is Lagrangian for each $m\in N$ (equivalently, if $N$ is two-dimensional and the pullback of $\omega$ by the inclusion is identically zero). A Lagrangian \textit{distribution} $L$ on $M$ is a smooth map $L:M\rightarrow Gr(2,TM)$ such that each $L_{m}:=L(m)$ is Lagrangian. A distribution $L$ is said to be integrable if the corresponding set of vectors is involutive, that is, if for each $m\in M$ and for each $X,Y\in L_{m}$ we have $[X,Y]\in L_{m}.$ By the Frobenius theorem, an involutive (Lagrangian) distribution locally defines a foliation of $M$ by (Lagrangian) submanifolds \cite[Sec. 2.3]{Morita}.

Introduced by Tomassini and Vezzoni in \cite{Tomassini-Vezzoni}, a \textit{generalized CY (Calabi-Yau) structure} on $M$ is a triple $(\omega,J,\varepsilon)$ such that 1) $J$ is an $\omega$-compatible almost complex structure, and 2) $\varepsilon$ is a nonvanishing $(2,0)$-form such that
\[
\varepsilon\wedge\bar{\varepsilon}=\omega^2/2\text{ and } d(\operatorname{Re}\varepsilon)=0.
\]
A submanifold $p:L\hookrightarrow M$ is \textit{special Lagrangian} with respect to a generalized CY structure $(\omega,J,\varepsilon)$ if it is Lagrangian and
\[
p^{\ast}(\operatorname{Im}\varepsilon)=0.
\]

If $J$ and $\varepsilon$ are left $G$-invariant, then a CY structure $(\omega,J,\varepsilon)$ induces an algebraic structure on the Lie algebra $\mathfrak{g}$ (denoted by the same symbols), and vice versa. We can therefore check that a left $G$-invariant Lagrangian foliation is special Lagrangian by checking the corresponding conditions in $\mathfrak{g}$.

\section{\label{sec:rep thry i} Representation theory of $\widetilde{G}$ (Part I): subordinate subalgebras}

That the symplectic geometry of $M$ is related to the algebraic structure of $\widetilde{G}$ becomes apparent after a careful analysis of the representation  theory of $\widetilde{G}$ using Kirillov's orbit method, which is ideally  suited to our situation since $\widetilde{G}$ is nilpotent (see \cite{Kirillov} for a thorough treatment of the orbit method). In fact, it is a seemingly innocuous choice, from a representation theoretic point of view, which provides  the connection: the choice of subordinate subalgebra.

In this section, we begin the orbit method analysis and describe explicitly the relevant subordinate subalgebras. Their connection with the symplectic geometry of $M$ will be taken  up in the next section. The orbit method analysis will then be completed in Section \ref{sec:rep thry ii}, where we return to the idea of $\vartheta$-functions on $M$.

The unitary dual of $\widetilde{G}$ is parameterized by the set of coadjoint orbits; among these, there is a family of $4$-dimensional orbits $\Omega_{\mu}:=Ad(\widetilde{G})^{\ast}(\mu\beta^U)$ parameterized by $\mu\in\mathbb{R}\setminus\{0\}.$ The orbit method is (among other things) an explicit algorithm which constructs a unitary irreducible representation of $\widetilde{G}$ for each coadjoint orbit. The first step in the orbit method algorithm is to find the coadjoint orbits and associated subordinate subalgebras; we recall their definition.

\begin{definition}
\label{def:ssa}
A subalgebra $\mathfrak{h}<\mathfrak{\tilde{g}}$ is \textbf{subordinate} to $\Omega_{\mu},$ or  $\Omega_{\mu}$-\textbf{subordinate}, if for any (and hence every) $\lambda\in\Omega_{\mu},$
\[
\left.\lambda\right\vert_{[\mathfrak{h},\mathfrak{h}]}=0
\]
and $\dim\mathfrak{h}$ is maximal among such subalgebras.
\end{definition}

Before we get to the technicalities of the unitary dual of $\widetilde{G}$, we make one final remark. Even from a representation theoretic point of view, the subordinate subalgebra plays a certain role which does not seem to have been observed: each choice of subalgebra subordinate to $\Omega_{\mu=k},~k\in2\mathbb{Z}$ leads to a different orthogonal decomposition $L_k^2(P)=4k^2 L^2(\mathbb{R}^2)$. This fact will become clear after we study periodizing maps in Section \ref{sec:rep thry ii}.

\subsection{{\label{subsec:reps of Gt}Subordinate subalgebras}}

Equivalence classes of unitary irreducible representations of $\widetilde{G}$ are in one-to-one correspondence with the coadjoint orbits of $\widetilde{G}$. Let $\mathfrak{h}$ be a $\Omega$-subordinate subalgebra for some coadjoint orbit $\Omega$. A character $\bar{\lambda}_{\Omega}$ of the connected analytic subgroup $H$ of $\widetilde{G}$ with Lie algebra $\mathfrak{h}$ is
\begin{equation}
\label{eqn:character}
h\in H\mapsto\bar{\lambda}_{\Omega}(h)=\exp\left(2\pi i\left\langle \lambda,\log h\right\rangle \right)\in U(1)
\end{equation}
where $\lambda\in\Omega$ is any point in the coadjoint orbit, and $\left\langle\cdot,\cdot\right\rangle$ denotes the canonical pairing of $\widetilde{\mathfrak{g}}^{\ast}$ with $\widetilde{\mathfrak{g}}$.

Since $\widetilde{G}$ is nilpotent, all of the unitary irreducible representations of $\widetilde{G}$ are induced from the characters of the analytic subgroups of $\widetilde{G}$ corresponding to the subordinate subalgebras; that is, given a subalgebra $\mathfrak{h}$ subordinate to $\Omega$ and the corresponding Lie subgroup $H$, a unitary irreducible representation of $\widetilde{G}$ is defined on $L^2(H\backslash\widetilde{G})$ by
\begin{equation}
\label{eqn:induced rep}
[\operatorname{Ind}_{H}^{\widetilde{G}}(g)f](x) =\bar{\lambda}_{\Omega}(h(x,g))f(xg),
\end{equation}
where $h(x,g)$ is the solution to the so-called master equation
\begin{equation}
\label{eqn:master eqn}
s(x)g=h(x,g)s(xg)
\end{equation}
for some choice of section $s:H\backslash\widetilde{G}\rightarrow\widetilde{G}$ (see \cite[Chap. 3]{Kirillov} for details). It follows from \eqref{eqn:master eqn} that $h$ satisfies the cocycle condition
\begin{equation}
\label{eqn:cocycle}
h(x,g_1g_2)=h(xg_1,g_2)h(x,g_1).
\end{equation}

\subparagraph{Assumption:} We will always choose $s:H\backslash\widetilde{G}\rightarrow\widetilde{G}$ so that $s(H)=\mathbf{0}.$

\noindent The task now is to enumerate the space of coadjoing orbits.

\begin{theorem}
\label{thm:coadjoint orbits}
The space of coadjoint orbits of $\widetilde{G}$  is:
\begin{itemize}
\item for each $\mu\in\mathbb{R}\setminus\{0\}$ a four-dimensional orbit through $(0,0,0,0,\mu),$
\item for each $\alpha_3 \in\mathbb{R}\setminus\{0\},\rho\in\mathbb{R}$ a two-dimensional orbit through $(0,0,\alpha_3,\rho,0),$ and
\item for each $(\alpha_1,\alpha_2,\rho)\in\mathbb{R}^3 $ a zero-dimensional orbit through $(\alpha_1,\alpha_2,0,\rho,0).$
\end{itemize}
Topologically, this space is $\mathbb{R}$ with the origin removed and replaced by a copy of $\mathbb{R}^2$ in which one axis is removed, each point of which  is replaced by another copy of $\mathbb{R}^2.$
\end{theorem}

\begin{proof}
Using the formula
\[
Ad^{\ast}((\mathbf{a},r,v))=\,^T{Ad((\mathbf{a},r,v))}^{-1}
\]
in the coordinates with respect to $\{X_1,X_2,X_3,T,U\}$ on $\mathfrak{g}$, and the dual coordinates $(\alpha_1,\alpha_2,\alpha_3,\rho,\mu)$ defined by $\{\beta^1,\beta^2,\beta^3,\beta^T,\beta^U\}$ on $\mathfrak{g}^{\ast},$ the coadjoint action of $\widetilde{G}$ is
\begin{multline}
Ad^{\ast}((\mathbf{a},r,v))(\alpha_1,\alpha_2,\alpha_3,\rho,\mu)\\
=(\alpha_1+a^2\alpha_3-a^3\mu,\alpha_2-a^1\alpha_3-(r+{\tfrac12}(a^1)^2)\mu,\alpha_3+a^1\mu,\rho+a^2\mu,\mu).
\end{multline}

The first statement of the theorem follows from the fact that if $\mu\neq0$, then the choice
\[
a^1=-\alpha_3/\mu,\ a^2=-\rho/\mu,\ a^3=\frac{\mu\alpha_1-\rho\alpha_3}{\mu^2},\ r=\frac{\alpha_3^2+2\alpha_2\mu}{2\mu^2}
\]
yields $Ad^{\ast}((\mathbf{a},r,v))(\alpha_1,\alpha_2,\alpha_3,\rho,\mu) =(0,0,0,0,\mu)$. The rest of the computations are similar.
\end{proof}

\smallskip
Observe that the center of $\widetilde{G}$ acts nontrivially only on the 4-dimensional orbits. Since it is the center of $\widetilde{G}$ which acts as $S^1$ on the fibers of the prequantum bundle $P,$ we expect, and it is indeed the case, that these orbits will play a prominent role in the harmonic analysis of $P$.

To construct the unitary irreducible representation associated to a coadjoint orbit $\Omega$ we must find a corresponding $\Omega$-subordinate subalgebra (Definition \ref{def:ssa}).

It is worth noting that any choice of subordinate subalgebra will do for the construction of a representation corresponding to $\Omega$, but there are many such choices. Although they induce equivalent representations of $\widetilde{G},$ different choices of subordinate subalgebra will induce \textit{different} decompositions of $L^2(P)$ into irreducible factors, and so we will take some care to enumerate here all such choices. Moreover, we will see in Section \ref{sec:lagr foliations} that the different choices of subordinate subalgebra reflect the symplectic geometry of the Kodaira--Thurston manifold.

We have three types of orbits to consider. The choice of subordinate subalgebra will only be relevant for the four-dimensional orbits, and so it is only in that case that we enumerate \textit{all} such choices.

\begin{theorem}
\label{thm:subordinate subalgebras}
(Subordinate Subalgebras)
\begin{itemize}
\item Corresponding to orbits of the form $\Omega=Ad(\widetilde{G})^{\ast}(\alpha_1,\alpha_2,0,\rho,0)$, there is a  unique $5$-dimensional subordinate subalgebra: $\tilde{\mathfrak{g}}.$
\item Associated to an orbit $\Omega=Ad(\widetilde{G})^{\ast}(0,0,\alpha_3,\rho,0)$, a choice of $4$-dimensional subordinate subalgebra is
\[
\mathfrak{h}_{\alpha_3,\rho}=\operatorname{span}_{\mathbb{R}}\{X_2,X_3,T,U\}.
\]
\item To the orbits $\Omega_{\mu}=Ad(\widetilde{G})^{\ast}(0,0,0,0,\mu),~\mu\neq0$, the following subalgebras are subordinate:
\begin{align*}
\mathfrak{h}^{c} &:=\mathbb{R}(X_1+cX_3)\oplus\mathbb{R}T\oplus\mathbb{R}U,~c\in\mathbb{R}\cup\{\infty\},\\
\mathfrak{h}^{b,d} &:=\mathbb{R}(X_1+bX_2+dT)\oplus\mathbb{R}(X_3-\frac1{b}T)\oplus\mathbb{R}U, ~b\in\mathbb{R}\cup\{\infty\}\setminus\{0\}, ~d\in\mathbb{R\cup\{\infty\}}\text{ and}\\
\mathfrak{h}^e &:=\mathbb{R}(X_2+eT)\oplus\mathbb{R}X_3\oplus\mathbb{R}U, ~e\in\mathbb{R}\cup\{\pm\infty\}.
\end{align*}
where $\mathfrak{h}^{e=\pm\infty}\simeq\mathfrak{h}^{c=\infty}:=\mathbb{R}X_3\oplus\mathbb{R}T\oplus\mathbb{R}U$ and $\mathfrak{h}^{e=0}\simeq\mathfrak{h}^{b=\infty,d}:=\mathbb{R}X_2\oplus\mathbb{R}X_3\oplus\mathbb{R}U.$ In particular, the $\Omega_{\mu}$-subordinate subalgebras are independent of $\mu$.
\end{itemize}
\end{theorem}

\begin{proof}
To verify that the given subalgebras are indeed subordinate, use the fact that
\[
\left\langle (\alpha_1,\alpha_2,\alpha_3,\rho,\mu),[(\mathbf{a},r,v),(\mathbf{b},s,u)]\right\rangle =\alpha_3(a^1b^2-a^2b^1)+\mu(a^3b^1-a^1b^3+b^2r-a^2s).
\]
That \textit{all} of the subalgebras subordinate to $Ad(\widetilde{G})^{\ast}(0,0,0,0,\mu)$ are the ones given is a corollary of Theorem \ref{thm:subord iff Lag}. One simply enumerates all of the Lagrangian subspaces of $\mathfrak{g}$ and intersects with the set of subalgebras of $\mathfrak{g}$.
\end{proof}

\smallskip
An important observation for what comes later (the proof is a straightforward computation using Theorem \ref{thm:subordinate subalgebras} and hence omitted):

\begin{lemma}
\label{lemma:ideal SSAs}
The family $\{\mathfrak{h}^e, e\in\mathbb{R}\cup\{\pm\infty\}\}$ consists of ideals, and these are the only ideal subordinate subalgebras. Moreover, the subalgebras $\mathfrak{h}^e$ are commutative.
\end{lemma}

For these reasons, computations work especially nicely if we choose one of the $\mathfrak{h}^e$ subalgebras, and so throughout the remainder, if we need to do a model computation, we will use $\mathfrak{h}^0:=\mathfrak{h}^{e=0}.$

\section{Lagrangian foliations\label{sec:lagr foliations}}

We turn our attention now to Lagrangian and special Lagrangian foliations and fibrations. First, we recall a generalization of the notion of special Lagrangian due to Tomassini and Vezzoni (see \cite{Tomassini-Vezzoni} for details). Then, we will show that the Lagrangian distributions associated to $\mathfrak{h}^e$ are in fact special Lagrangian foliations, and for certain values of $e$, these foliations are fibrations by tori.

The connection between the representation theory and symplectic geoemetry of our setup is a consequence of the following simple result.

\begin{lemma}
For $X\in\tilde{\mathfrak{g}},$ let $X_0\in\mathfrak{g}$ be the $g$-orthogonal projection of $X$ onto $\mathfrak{g}$ \eqref{eqn:g-proj gt->g}. Then
\[
\mu\beta^U([X,Y])=-2\pi\mu\,\omega(X_0^{L},Y_0^{L}).
\]
\end{lemma}

\begin{proof}
Let $X=x^jX_{j}+x^TT+x^UU$ and $Y=y^jY_{j}+y^TT+y^UU$. Then
\[
\mu\beta^U([X,Y])=-\mu(x^1y^3-x^3y^1+x^2y^T-x^Ty^2)=-2\pi\mu\omega(X_0^{L},Y_0^{L}).
\]
\end{proof}

\smallskip
The $\Omega_{\mu}$-subordinate subalgebras listed in Theorem \ref{thm:subordinate subalgebras} are $3$-dimensional. It then follows from the general theory of the orbit method that \textit{all }$\Omega_{\mu}$-subordinate subalgebras are $3$-dimensional (to avoid a circular argument, it is important  that we do not assume here that Theorem \ref{thm:subordinate subalgebras} lists \textit{all} of the $\Omega_{\mu}$-subordinate subalgebras).

\begin{theorem}
\label{thm:subord iff Lag}
A subalgebra $\mathfrak{h}\subseteq\tilde{\mathfrak{g}}$ is $\Omega_{\mu}$-subordinate if and only if $\mathfrak{h}=L\oplus\mathbb{R}U$ for some Lagrangian subspace $L\subset\mathfrak{g}$.
\end{theorem}

\begin{proof}
First, suppose $\mathfrak{h}=L\oplus\mathbb{R}U$ is a subalgebra for some Lagrangian $L$. Then
\[
\mu\beta^U([X_0+aU,Y_0+bU])=\omega(X_0,Y_0)=0
\]
and $\mathfrak{h}$ is of maximal dimension; hence $\mathfrak{h}$ is  subordinate.

In the other direction, suppose that $\mathfrak{h}$ is subordinate. Then since $[\mathbb{R}U,\mathfrak{\tilde{g}}]=\{0\},$ we must have $\mathbb{R}U\subseteq\mathfrak{h}$. Let $L\subset\mathfrak{g}$ be the projection of $\mathfrak{h}$ onto $\operatorname{span}\{X_1,X_2,X_3,T\}$.  Then
\[
\omega(X_0^{L},Y_0^{L})=-\frac1{2\pi}\beta^U([X,Y])=0
\]
so that $L$ is Lagrangian as desired.
\end{proof}

\smallskip
Be careful that it is necessary that $\mathfrak{h}$ is a subalgebra in either direction; in fact, there is a $3$-dimensional family of Lagrangian subspaces\footnote{The family is $\{L=\mathbb{R}(X_1+aX_3+bT)\oplus \mathbb{R}(X_2+bX_3+cT):a,b,c,\in\mathbb{R}\}.$} $L$ such that $L\oplus\mathbb{R}U$ is not a subalgebra, and so the correspondence $\mathfrak{h}$ $\mapsto L$ is only injective.

Each subspace $L\subset\mathfrak{g}$ defines a left-invariant distribution on $M$. If $L$ is a subalgebra, then this distribution is integrable. If $L$ is Langrangian, then so is the corresponding left-invariant distribution, and hence each $\Omega_{\mu}$-subordinate subalgebra $\mathfrak{h}$ induces an integrable Lagrangian distribution on $M,$ that is, a Lagrangian foliation.

\begin{theorem}
\label{thm:Lag tori}
The foliation induced by $\mathfrak{h}^e$ is a fibration of $M$ by Lagrangian tori if and only if $e=0,\pm\infty$.
\end{theorem}

\begin{proof}
Let $T^e$ be the real analytic subgroup of $G$ with Lie algebra $L,$ where $\mathfrak{h}^e=L\oplus\mathbb{R}U$. Then $T^e$ is diffeomorphic to $\mathbb{R}^2$. The leaves of the foliation induced by $\mathfrak{h}^e$ are the orbits of $T^e.$ One easily checks that if $e=0,\pm\infty$, the $T^e$-orbits in $M$ are all tori. Moreover, if $e\neq0,\pm\infty$ then the $T^e$-orbit through $[x,y,z,t]$ is compact if and only if $x$ and $e$ are linearly dependent over $\mathbb{Q}$.
\end{proof}

\smallskip
Among the $\Omega_{\mu}$-subordinate subalgebras $\mathfrak{h}$, there is a certain family $\mathfrak{h}^e,~e\in\mathbb{R\cup\{\pm\infty\}}$ which is distinguished by the following results. (See Section \ref{subsec:Lagr in M} for  the definition of \textit{special Lagrangian}).

\begin{theorem}
\label{thm:sLag}
For each $e\in\mathbb{R}$ there exists a left-invariant CY structure $(\omega,J_e,\varepsilon_e)$ on $M$ with respect to which the left $G$-invariant Lagrangian foliation of $M$ induced by the subordinate subalgebra $\mathfrak{h}^e$ is special Lagrangian.
\end{theorem}

\subparagraph{Remark}
The special Lagrangian torus defined by $\mathfrak{h}^{e=0}$ was discovered by Tomassini and Vezzoni \cite{Tomassini-Vezzoni}.\hfill$\square$

\medskip
\begin{proof}
The set of $\omega$-compatible compex structures on a symplectic vector space of real dimension four can be parameterized by the generalized upper half-space \cite[Sec. 2.5]{McDuff-Salamon}
\[
\mathfrak{H}_{+}:=\{\Omega\in M_2(\mathbb{C}) : ~^T\Omega=\Omega,~\operatorname{Im}\Omega>0\}.
\]
Given a point $\Omega\in\mathfrak{H}_{+}$, the corresponding $\omega$-compatible complex structure is
\[
J_{\Omega}=
\begin{pmatrix}
\Omega_1\Omega_2^{-1} & -\Omega_2-\Omega_1\Omega_2^{-1}\Omega_1\\
\Omega_2^{-1} & -\Omega_2^{-1}\Omega_1
\end{pmatrix}.
\]
Hence, the complex structure $J_e$ on $\mathfrak{g}$ corresponding to the point
\[
\Omega_e=
\begin{pmatrix}
(1+2|e| )(\frac{-|e| }{1+|e| }+i) & \sqrt{\left\vert  e\right\vert }(-1+i)\\
\sqrt{|e| }(-1+i) & \frac1{1+2|e|}(-e+i(1+e))
\end{pmatrix}
\]
is
\[
J_e=
\begin{pmatrix}
0 & -\frac{\sqrt{|e|}(1+2|e|)}{1+|e| } & -4|e| -\frac1{1+|e|} & -\frac{\sqrt{e}(1+2|e|)}{1+|e| }\\
-\frac{\sqrt{|e|}}{1+2|e|} & 0 & -\frac{\sqrt{|e| }(1+2|e|)}{1+|e|} & -1\\
\frac{1+|e|}{1+2|e|} & -\sqrt{|e|} & 0 & \frac{\sqrt{|e|}}{1+2|e|}\\
-\sqrt{|e|} & 1+2|e|  & \frac{\sqrt{|e|}(1+2|e|)}{1+|e|} & 0
\end{pmatrix}.
\]

Define the $(2,0)$-form (with respect to $J_e$)
\begin{align*}
\varepsilon_e  & =\pi i(\beta^1\wedge\beta^2 +\sqrt{|e|}(1+i)\beta^1\wedge\beta^3 +\left(\frac{|e|+i(1+|e|)}{1+2|e|}\right)\beta^1\wedge\beta^4\\
&+\left(\frac{-|e|(1+2|e|)}{1+|e| }-i(1+2|e| )\right)\beta^2\wedge\beta^3 -\sqrt{|e|}(1+i)\beta^2\wedge\beta^4 -\left(\frac{1+2|e|}{1+|e|}\right)\beta^3\wedge\beta^4)
\end{align*}

It is now routine (though somewhat tedious) to check that the foliation of $M$ induced by $\mathfrak{h}^e$ is special Lagrangian; we leave the details to the reader (who may find it useful to recall that $d\beta^3=-\beta^1\wedge\beta^2$ and $d\beta^j=0,$ $j\neq3$).
\end{proof}

\smallskip
At $e=\pm\infty$, the complex structure degenerates; this is a geometric feature of the foliation induced by $\mathfrak{h}^{e=\pm\infty}.$ Indeed, given an arbitrary $\Omega\in\mathfrak{H}_{+}$, one may write any left-invariant $(2,0)$-form $\alpha$ (with respect to $J_{\Omega}$) in terms of the components of $\Omega=\begin{pmatrix}a & b\\b & d\end{pmatrix}$: for some $f\in C^{\infty}(M)$, we obtain
\[
\alpha=\pi if\left(  \beta^1\wedge\beta^2 +\overline{(ad-b^2)}\beta^3\wedge\beta^4-\bar{b}\beta^1\wedge\beta^3 -\bar{d}\beta^1\wedge\beta^4 +\bar{a}\beta^2\wedge\beta^3 +\bar{b}\beta^2\wedge\beta^4\right).
\]
The condition that $p^{\ast}(\operatorname{Im}\varepsilon)=0$ then implies the vanishing of the imaginary part of the coefficient of $\beta^3\wedge \beta^4$, that is, $\operatorname{Im}(\det\Omega)=0$. Hence, $\Omega$ lies on  the boundary of $\mathfrak{H}_{+}$ so that the foliation is not special Lagrangian with respect to any CY structure.

\begin{corollary}
The foliation induced by $\mathfrak{h}^e$ is by special Lagrangian tori if and only if $e=0.$
\end{corollary}

\section{\label{sec:rep thry ii} Representation theory of $\widetilde{G}$ (Part II) : $\vartheta$-functions and the decomposition of $L^2(M,\ell^{\otimes k})$}

We return now to the study of the unitary dual of $\widetilde{G}$, and in particular the decomposition of $L^2(P)$ into unitary irreducible representations of $\widetilde{G}$. We begin by showing that only those representations corresponding $V_k$ to certain integral $4$-dimensional coadjoint orbits contribute nontrivially to the decomposition; in particular, we show that
\[L^2(P)\simeq\bigoplus_{k\in\mathbb{Z}}4k^2 V_k\oplus L^2(M).\]
Next, we will compute the multiplicities appearing in the decomposition of $L^2(P)$. Finally, we will construct periodizing maps---the analogues for the Kodaira--Thurston manifold of the Weil--Brezin map---which, for each choice of subordinate subalgebra, achieve an orthogonal decomposition of each invariant subspace of $L^2(P)$ into irreducible factors. Finally, we will investigate the pseudoperiodicity of the periodizing maps.

We are interested in the space of $L^2$-sections of the $k$-th tensor power of the prequantum circle bundle $P:=\widetilde{\Gamma}\backslash\widetilde{G}$ over $M,$ for $k\geq1$. Such a section is equivalent to a $k$-equivariant function $f\in L^2(P),$ that is, one which is equivariant with respect to the  circle action on the fibers of $P$ (see the discussion following Lemma \ref{lemma:preQ circle bundle})
\begin{equation}
\label{eqn:k-equiv}
f\left(pe^{2\pi i\theta}\right)=e^{-4\pi ik\theta}f(p);
\end{equation}
The isotypical subspace of $L^2(P)$ consisting of $k$-equivariant functions is denoted by $L_k^2(P).$

Of course, the circle action on the fibers of $P$ is just the action of the center of $\widetilde{G}$ (or, more precisely, $K$) on $P$.

\begin{lemma}
\label{lemma:reps w nontrivial center action}
The representations $\pi_{\mu}$ of $\widetilde{G}$ corresponding to the coadjoint orbits $Ad(\widetilde{G})^{\ast}(0,0,0,0,\mu),~\mu\neq0$, are the only unitary irreducible representations which are nontrivial on the center of  $\widetilde{G}.$ The equivalence class of such unitary irreducible representation is uniquely determined by its value on the center of $\widetilde{G},$ which is
\begin{equation}
\label{eqn:center action on P}
[\pi_{\mu}]((\mathbf{0},v))f=e^{2\pi i\mu v}f.
\end{equation}
\end{lemma}

\begin{proof}
We first show that the represenations $\pi_{\mu}$ have the desired properties. We will compute in the model case $\mathfrak{h}^{e=0}=\mathbb{R}X_2\oplus\mathbb{R}X_3\oplus\mathbb{R}U$. By definition \eqref{eqn:induced rep},
\begin{align*}
\left[\pi_{\mu}((\mathbf{0},v))f\right](H^0\mathbf{g}) & =\left[\operatorname{Ind}_{H^0}^{\widetilde{G}}((\mathbf{0},v))f\right](H^0\mathbf{g})\\
& =\exp\left\{2\pi i\mu\left\langle(0,0,0,0,\mu),(0,0,0,0,v)\right\rangle\right\} f(H^0\mathbf{g})\\
& =e^{2\pi i\mu v}f(H^0\mathbf{g}).
\end{align*}

Similar computations show that the representations associated to the other coadjoint orbits are trivial on the center; for example, the representation associated to $Ad(\widetilde{G})^{\ast}(0,0,\alpha_3,\rho,0)$, evaluated at the point $(\mathbf{0},v)$, yields
\[
\exp\left\{2\pi i\mu\left\langle (0,0,\alpha_3,\rho,0),(0,0,0,0,v)\right\rangle \right\} f=f.
\]
\end{proof}

\smallskip
Hence, the isotypical subspace $L_k^2(P),\ k\in\mathbb{Z}\setminus\{0\}$ is also isotypical with respect to the action of $\widetilde{G},$ and decomposes as a direct sum of unitary irreducible representations corresponding  to $\mu=-2k.$ There is no canonical way of choosing a canonical decomposition  of the isotypical subspace $L_k^2(P)$ into irreducible representations. On the other hand, we may compute the multiplicity with which the representation  $\pi_{-2k}$ appears in $L_k^2(P)$ unambiguously. Also, it will turn out  that each choice of subalgebra subordinate to $(0,0,0,0,-2k)\in\mathfrak{g}^{\ast}$ will induce a decomposition of $L_k^2(P)$.

Each $\Omega_{\mu}$-subordinate subalgebra $\mathfrak{h}$ is $3$-dimensional, so $H:=\exp(\mathfrak{h})$ is also $3$-dimensional. The unitary irreducible representation induced by $\mathfrak{h}$ is
\[
\operatorname*{Ind}\nolimits_{H}^{\widetilde{G}} :\widetilde{G}\rightarrow\operatorname*{End} (V_k=L^2(H\backslash\widetilde{G})).
\]
But $H\backslash\widetilde{G}\simeq\mathbb{R}^2,$ and since $\widetilde{G}$ is unimodular the measure on $H\backslash\widetilde{G}$ is identified with the Lebesgue measure on $\mathbb{R}^2,$ so $V_k\simeq L^2(\mathbb{R}^2,dx\,dt)$ \cite[Sec. V.2.2]{Kirillov}. We compute $\operatorname*{Ind}\nolimits_{H}^{\widetilde{G}}$ in detail in the Example at the end of this section.

First, we consider the isotypical subspace more precisely. Let $V_k=L^2(\mathbb{R}^2,dx\,dy)$ denote the representation space for $\pi_{-2k}:\widetilde{G}\rightarrow End(V_k),$ and consider the evaluation map
\[
Hom_{\widetilde{G}}(V_k,L_k^2(P))\otimes V_k\rightarrow L^2(P)
\]
where $Hom_{\widetilde{G}}(V_k,L_k^2(P))$ is the space of $\widetilde{G}$-equivariant maps from $V_k$ to $L^2(P)$. The image of this map is the isotypical subspace corresponding  to $\pi_{-2k}$. Since $\pi_{-2k}$ is uniquely determined by its value on the center of  $\widetilde{G}$ (Lemma \ref{lemma:reps w nontrivial center action}), which by \eqref{eqn:center  action on P} is exactly the $k$-equivariance condition \eqref{eqn:k-equiv}, this image is  precisely the isotypical subspace $L_k^2(P).$

The isotypical subspace $L_k^2(P)$ therefore decomposes into copies of $V_k$, that is,
\begin{equation}
\label{eqn:L2k step1 decomp}
L_k^2(P)\simeq V_k\oplus\cdots\oplus V_k=m(\pi_{-2k},L_k^2(P))V_k,
\end{equation}
where $m(\pi_k,L_k^2(P))$ denotes the multiplicity with which $(\pi_k,V_k)$ appears in $L_k^2(P)$. As remarked in the Introduction, Brezin proved the existence of the  decomposition \eqref{eqn:L2k step1 decomp}) in \cite{Brezin}, where he also gives a procedure for achieving the decomposition. Brezin's procedure is somewhat different from our approach, which is based on Richardson's periodizing maps \cite{Richardson}.

\begin{theorem}
\label{thm:multiplicity is k^2}
For $k\in\mathbb{Z}\setminus\{0\},$ the multiplicity with which $(\pi_{-2k},V_k)$ appears in $L_k^2(P)$ is
\[
m(\pi_{-2k},L_k^2(P))=4k^2.
\]
\end{theorem}

A multiplicity formula for the decomposition of the $L^2$-space of a general nilmanifold was discovered by Moore \cite{Moore} and independently by Richardson \cite{Richardson}.  Richardson's proof of this formula will have important consequences later, so we recall the  setup here.

\smallskip As described in Section \ref{sec:rep thry i}, to each $\lambda\in\widetilde{\mathfrak{g}}^{\ast}$ and choice of $\lambda$-subordinate subalgebra $\mathfrak{h}_{\lambda}$ there is associated a character $\bar{\lambda}:H_{\lambda}=\exp(\mathfrak{h}_{\lambda})\rightarrow U(1)$ given by
\[
\bar{\lambda}(h)=e^{2\pi i\left\langle \lambda,\log(h)\right\rangle}.
\]
The pair $(\bar{\lambda},H_{\lambda}),$ called a \textit{maximal character}, induces a unitary irreducible representation $\pi_{\lambda}=\operatorname{Ind}_{H_{\lambda}}^{\widetilde{G}}$ given by equation \eqref{eqn:induced rep}.

The group $\widetilde{G}$ acts on the set of pairs $\{(\bar{\lambda},H_{\lambda})\}$ by $(\bar{\lambda},H_{\lambda})\cdot g=(\bar{\lambda}^{g},^{g^{-1}}H_{\lambda}),$ where
\[
^{g^{-1}}H_{\lambda}:=g^{-1}Hg\text{\ \ and\ \ }\bar{\lambda}^{g}(h):=\overline{\lambda}(ghg^{-1}).
\]
A pair $(\bar{\lambda},H_{\lambda})$ is called a \textit{rational maximal character} if  $\dim_{\mathbb{R}}\mathfrak{h}_{\lambda} =\dim_{\mathbb{Q}}(\mathfrak{h}_{\lambda} \cap\log(\widetilde{\Gamma}))$ and $\lambda:\mathfrak{h}_{\lambda}\cap\log(\widetilde{\Gamma})\rightarrow\mathbb{Q}. $ A rational  maximal character is called an \textit{integral point} if $\bar{\lambda}(\widetilde{\Gamma}\cap H_{\lambda})=1.$

The keys to the proof of the Moore--Richardson formula (Theorem \ref{thm:Richardson mult}, below) are
\begin{enumerate}
\item $\pi_{\lambda}$ appears with multiplicity $m(\pi_{\lambda},L^2(P))>0$ if and only if the orbit $(\bar{\lambda},H_{\lambda})\cdot\widetilde{G}$ contains and integral point, and
\item if $\gamma\in\widetilde{\Gamma}$ and $(\bar{\lambda},H_{\lambda})$ is an integral point, then $(\bar{\lambda},H_{\lambda})\cdot\gamma$ is also an integral point.
\end{enumerate}

Moreover, Richardson associates to each integral point $(\bar{\lambda},H_{\lambda})$ an invariant subspace of $L^2(P)$. Both $(\bar{\lambda},H_{\lambda})$ and $(\bar{\lambda},H_{\lambda})\cdot\gamma$ induce the same invariant subspace, and if $(\bar{\lambda},H_{\lambda})$ and $(\bar{\lambda}^{\prime},H_{\lambda^{\prime}})$ are integral points in different $\widetilde{\Gamma}$-orbits, then the induced invariant subspaces are orthogonal. These subspaces are described in the next section. We may now deduce the Moore--Richardson multiplicity formula \cite{Moore},\cite{Richardson}.

\begin{theorem}
\label{thm:Richardson mult}Let $\left[(\bar{\lambda},H_{\lambda})\cdot\widetilde{G}\right]_{\mathbb{Z}}$ denote the set of integral points in the $\widetilde{G}$-orbit $(\bar{\lambda},H_{\lambda})\cdot\widetilde{G}$.
Then
\[
m(\pi_{\lambda},L^2(P)) =\#\left\{\left.\left[(\bar{\lambda},H_{\lambda})\cdot\widetilde{G}\right]_{\mathbb{Z}}\right/ \widetilde
{\Gamma}\right\}.
\]
\end{theorem}

To use the Moore--Richardson formula, we first need a lemma (for which we also find a use in  Section \ref{subsec:Laplacians}). Recall that $\mathfrak{h}^{e=0}:=\mathbb{R}X_2\oplus\mathbb{R}X_3\oplus\mathbb{R}U$ is $Ad(\widetilde{G})^{\ast}(0,0,0,0,\mu)$-subordinate for every $\mu\neq0$ (Theorem \ref{thm:subordinate subalgebras}). The corresponding analytic subgroup of $\widetilde{G}$ is
\[
H^0=\{(0,h_2,h_3,0,h_5)\in\widetilde{G}\}.
\]
Let $\bar{\lambda}_k:H^0\rightarrow U(1)$ be the character
\[
\bar{\lambda}_k(0,h_2,h_3,0,h_5)=\exp\{-4\pi ikh_5\}.
\]
Then $(\bar{\lambda}_k,H^0)$ is an integral point if and only if $\bar{\lambda}_k(\widetilde{\Gamma}\cap H^0)=1$, which implies $k\in\mathbb{Z}$.

\begin{lemma}
\label{lemma:integral points}The integral points of the orbit $\Omega=Ad(\widetilde{G})^{\ast}(0,0,0,0,-2k)$, $k\in\mathbb{Z}\setminus\{0\}$, with respect to the $\Omega$-subordinate subalgebra $\mathfrak{h}^{e=0},$ are $(\bar{\lambda}_k^{m,n},H^0),$ where $m,n\in\mathbb{Z}$ and
\[
\bar{\lambda}_k^{m,n}((0,y,z,0,v)):=\exp\{-4\pi i\,kv-2\pi i(my-nz)\}.
\]
\end{lemma}

\begin{proof}
We need first the action of $\widetilde{G}$ on the maximal characters. The situation is quite simple here: $\mathfrak{h}_k^{e=0}$ is an ideal, which implies $^{g^{-1}}H^0=H^0$ for all $g\in\widetilde{G}.$ With $h=(0,h_2,h_3,0,h_5)$ and $g=(x_0,y_0,z_0,t_0,u_0),$ we only need to compute $\bar{\lambda}_k^{g},$ where $\bar{\lambda}_k(h)=\exp\{-4\pi ikh_5\}$:
\begin{equation}
\bar{\lambda}_k^{g}(h)=\exp\left\{-4\pi i\,k\left(h_5+h_2\left(t_0-\frac{x_0^2}2\right)-h_3x_0\right)\right\}.
\end{equation}

The $\widetilde{G}$-orbit through $(\bar{\lambda}_k,H^0)$ is
\[
(\bar{\lambda}_k,H^0)\cdot\widetilde{G}=\left\{(h\mapsto\exp\left\{-4\pi ik\left(h_5+\delta h_2-\varepsilon h_3\right)\right\},H^0):\delta,\varepsilon\in\mathbb{R}\right\}.
\]
The set of integral points in this orbit is
\[
\{(h\mapsto\exp\left\{-4\pi ik\left(h_5+\delta h_2-\varepsilon h_3\right)\right\},H^0):\exp\left\{-4\pi ik \left(  h_5+\delta h_2-\varepsilon h_3\right)\right\}=1\,\text{for all }h\in H^0\cap\widetilde{\Gamma}\}.
\]
But $\exp\left\{-4\pi ik\left(h_5+\delta h_2-\varepsilon h_3\right)\right\}=1$ for all $h_2,h_3\in\mathbb{Z}$ and $h_5\in\frac12\mathbb{Z}$ if and only if $k\in\mathbb{Z}$ and $\delta,\varepsilon\in\frac1{2k}\mathbb{Z}$, i.e., if and only if there exist integers $m,n\in\mathbb{Z}$ such that $\delta=\frac{m}{2k}$ and $\varepsilon=\frac{n}{2k}.$ Hence,
\[
\left[(\bar{\lambda}_k,H)\cdot\widetilde{G}\right]_{\mathbb{Z}} =\left\{(\bar{\lambda}_k^{m,n},H^0):m,n\in\mathbb{Z}\right\}.
\]
\end{proof}

\smallskip
\begin{proof}[Proof of Theorem \ref{thm:multiplicity is k^2}]
For simplicity, we will compute the multiplicity of $\pi_{-2k}^{e=0}.$ Let  $\lambda_k=(0,0,0,0,-2k)$ for $k\neq0$. We need to count the number of $\widetilde{\Gamma}$-orbits in the set $\left[(\bar{\lambda}_k,H_{\lambda_k})\cdot\widetilde{G}\right]_{\mathbb{Z}}$ of integral points.

To find the $\widetilde{\Gamma}$-orbits in $[(\bar{\lambda}_k,H)\cdot\widetilde{G}]_{\mathbb{Z}}$, let $\gamma=(x_0,y_0,z_0,t_0,u_0)\in\widetilde{\Gamma}$ and $h=(0,y,z,0,v).$ Then
\begin{align*}
(\bar{\lambda}_k^{m,n}\cdot\gamma)(h) &= \exp\left\{-4\pi ikv-2\pi i\left[y\left( m-nx_0+2k(t_0-\frac{x_0^2}2)\right)+z(n+2kx_0)\right]\right\} \\
& =\bar{\lambda}_k^{m-nx_0+2k(t_0-x_0^2/2),n+2kx_0}(h).
\end{align*}
This defines an action of $\mathbb{Z}^2$ on $\left[(\bar{\lambda}_k,H)\cdot \widetilde{G}\right]_{\mathbb{Z}} \simeq(\frac1{2k}\mathbb{Z})^2$:
\[
(x_0,t_0)\cdot(m,n)=(m-nx_0+2k(t_0-x_0^2/2),n+2kx_0).
\]
It is not hard to show that a fundamental domain is $\left\{(\bar{\lambda}_k^{m,n},H^0):m,n=0,1,\dots,2k-1\right\}$. In particular, $\#\left\{\left.\left[(\bar{\lambda}_k,H)\cdot\widetilde{G}\right]_{\mathbb{Z}}\right/ \widetilde{\Gamma}\right\} =\#\left\{(\frac1{2k}\mathbb{Z})^2/\mathbb{Z}^2\right\}=4k^2.$ Figure \ref{fig:integral points} below depicts the $\mathbb{Z}^2$-orbits in $\left(\frac1{2k}\mathbb{Z}\right)^2$ and the images of the fundamental domain under this $\mathbb{Z}^2$-action for $k=3$.
\end{proof}

\smallskip
\begin{figure}[h]
\begin{center}
\includegraphics[width=2in]{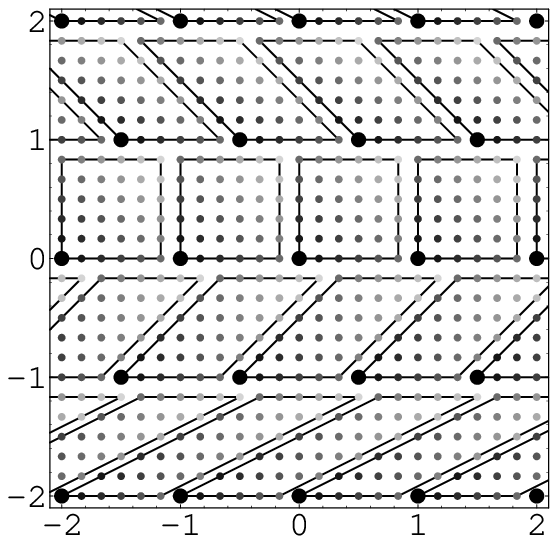}
\caption{$\widetilde{\Gamma}$-orbits of integral points and the fundamental
domain for $k=3$}
\label{fig:integral points}
\end{center}
\end{figure}

\smallskip
\subparagraph{Example}
We will find it useful to have explicit expressions for the representations induced by the integral points $(\bar{\lambda}_k^{m,n},H^0),~m,n=0,1,\dots,2k-1,~k\in\mathbb{Z}\setminus\{0\}$. To compute the induced representation, we need to solve the master equation \eqref{eqn:master eqn}. Recall that $\mathfrak{h}^{e=0}=\mathbb{R}X_2\oplus\mathbb{R}X_3\oplus\mathbb{R}U$ and hence that $H^0=\{(0,y,z,0,v)\in\widetilde{G}\}$. Since each coset in $H^0\backslash\widetilde{G}$ can be written in the form $H^{\infty}(x,0,0,t,0),$ we can identify $H^0\backslash\widetilde{G}$ with $\mathbb{R}^2$.

In the induction procedure, we use the section $s:H^0\backslash\widetilde{G}\rightarrow\widetilde{G}$ given by $s(H^{\infty}(x,0,0,t,0))=(x,0,0,t,0).$ The master equation is then
\[
(x,0,0,t,0)\cdot(a,b,c,r,v)=(0,h_2,h_3,0,h_5)\cdot(x^{\prime},0,0,t^{\prime},0).
\]
The solution is
\begin{gather*}
h_2=b,\quad h_3=c+b(x+a),\quad h_5=v+bt-c(x+a)-abx-\frac{b}2(x^2+a^2),\\
x^{\prime}=x+a,\quad t^{\prime}=t+r.
\end{gather*}

Again using the section $s$, we identify $(H^0\cap\widetilde{\Gamma})\simeq\mathbb{Z}^2$. The Haar measure on $\widetilde{G}$ descends to the Lebesgue measure on $\mathbb{R}^2\simeq H^0\backslash\widetilde{G}.$

The unitary irreducible representation $\pi_{-2k}^{m,n}:\widetilde{G}\rightarrow L^2(H^0\backslash\widetilde{G})\simeq L^2(\mathbb{R}^2,dx\,dt)$ associated to the coadjoint orbit $Ad(\widetilde{G})^{\ast}(0,0,0,0,-2k)$ and the subordinate subalgebra $\mathfrak{h}^{e=0}=\mathbb{R}X_2\oplus\mathbb{R}X_3\oplus\mathbb{R}U$, induced from the character $\bar{\lambda}_k^{m,n}(h)$ described in Lemma \ref{lemma:integral points}, is
\begin{align*}
\left(\pi_{-2k}^{m,n}(a,b,c,r,v)\right) f(x,t) &= \bar{\lambda}_k^{m,n}(0,h_2,h_3,0,h_5)f(x+a,t+r)\\
&=e^{-4\pi ik(v+bt-c(x+a)-abx-\frac{b}2(x^2+a^2))}e^{-2\pi i(mb-n(c+b(x+a)))}f(x+a,t+r).
\end{align*}
\hfill$\square$

\subsection{\label{subsec:periodizing maps} Periodizing Maps}

In this section, we describe the analogue $\Theta_k^j:L^2(\mathbb{R}^2)\rightarrow L_k^2(P:=\widetilde{\Gamma}\backslash\widetilde{G})$ of the Weil--Brezin map (discussed in the Introduction) for the Kodaira--Thurston manifold; both maps are instances of a very general  construction due to Richardson which we now describe.

Let $(\bar{\lambda},H_{\lambda})$ be an integral point for $\lambda\in \Omega=Ad(\widetilde{G})^{\ast}(0,0,0,0,\mu)$, which is possible only if $\mu=-2k\in2\mathbb{Z}$. To prove the multiplicity formula (Theorem \ref{thm:Richardson mult}), Richardson constructs a periodizing map\footnote{Richardson's construction, in the case of  $T^2=\mathbb{R}^2/\mathbb{Z}^2$, is the classical $\vartheta$-map (see, for example, \cite{Auslander-Tolimieri} for the relevant definitions).} $\Theta_k^{(\lambda)}:L^2(H_{\lambda}\backslash\widetilde{G})\rightarrow L_k^2(P)$ from the induced representation space to the $k$-isotypical subspace of $L^2(P)$. The image of  $\Theta_k^{(\lambda)}$ is an irreducible subspace, and two integral points in the same  $\widetilde{G}$-orbit induce periodizing maps with the same image. Moreover, the images of two periodizing maps are orthogonal in $L_k^2(P)$ if the associated integral points lie in distinct $\widetilde{G}$-orbits.

Since each function in $L_k^2(P)$ corresponds to a section of $\ell^{\otimes k},$ each map $\Theta_k^{(\lambda)}$ corresponds to a map
\[
\theta_k^{(\lambda)}:L^2(H_{\lambda}\backslash\widetilde{G})\rightarrow L^2(M,\ell^{\otimes k}).
\]
The prequantum line bundle $\ell^{\otimes k}$ lifts to a line bundle $\check{\ell}^{\otimes k}\rightarrow G\simeq\mathbb{R}^4$. After trivializing $\check{\ell}$, to each $f\in L^2(H_{\lambda}\backslash\widetilde{G})$ there is associated a section $\theta_k^{(\lambda)}f$ and hence a function
\[
\vartheta_k^{(\lambda)}f:G\rightarrow\mathbb{C}.
\]
The function $\vartheta_k^{(\lambda)}f$ is square-integrable on any fundamental domain $FD_{\Gamma_0\backslash G}$ of $\Gamma_0\backslash G $; denote the set of such maps by
\[
L^2(FD_{\Gamma_0\backslash G})=\left\{f:G\rightarrow\mathbb{C}\ \big\vert\ \int\nolimits_{FD_{\Gamma_0\backslash G}}\left\vert f\right\vert ^2d^4x<\infty\right\}.
\]
The maps $\theta_k^{(\lambda)}$ were the maps referred to in the Introduction, but we will  henceforth find it easier to work with $\Theta_k^{(\lambda)}$ and later with $\vartheta_k^{(\lambda)}.$

Although Richardson does not use the language of induced representations to do so, the periodizing maps $\Theta_k^{(\lambda)}$ can be described succinctly in terms of induced  representations, where it becomes transparent that a periodizing map is essentially a sum over  the remaining nonperiodic directions (i.e., over that portion of $\widetilde{\Gamma}$ which  lies outside of $\widetilde{\Gamma}\cap H$).

\begin{definition}
Let $(\bar{\lambda},H^0)$ be an integral point of a coadjoint orbit $\Omega=Ad(\widetilde{G})^{\ast}(0,0,0,0,-2k)$. The periodizing map $\Theta_k^{(\lambda)}:L^2(H^0\backslash\widetilde{G})\rightarrow L_k^2(P)$ associated to $(\bar{\lambda},H^0)$ is
\[
\left(\Theta_k^{(\lambda)}(f)\right)(g):=\sum_{[\gamma]\in(\Gamma\cap H_{\lambda})\backslash\Gamma} \left[\operatorname{Ind}_{H_{\lambda}}^{\widetilde{G}}([\gamma]g)f\right](H).
\]
\end{definition}

It is not hard to show that $\operatorname{Ind}_{H_{\lambda}}^{\widetilde{G}}$ is constant on right $(\Gamma\cap H_{\lambda})$-cosets, so that $\Theta_k^{(\lambda)}$ is well-defined. In \cite{Richardson}, Richardson also shows that $\Theta_k^{(\lambda)}$ is unitary up to a constant; specifically, that
\[
\left\langle f,g\right\rangle _{L^2(H_{\lambda}\backslash\widetilde{G})} =vol((\widetilde{\Gamma}\cap H^0)\backslash H^0)\left\langle \Theta_k^j f,\Theta_k^j g\right\rangle _{L^2(P)},
\]
and moreover that $\Theta_k^{(\lambda)}$ intertwines the right actions of $\widetilde{G}$ on $L^2(H^0\backslash\widetilde{G})$ and $L_k^2(P)$.

Combining the multiplicities given by Theorem \ref{thm:multiplicity is k^2} with the fact that the images of $\Theta_k^{(\lambda)}$ and $\Theta_k^{(\lambda^{\prime})}$ are orthogonal if $\lambda$ and $\lambda^{\prime}$ lie in distinct $\widetilde{G}$-orbits, we have the following  result.

\begin{corollary}
\label{cor:L2 decomp}
For each $k\in\mathbb{Z}\setminus\{0\},$ and each choice of  $Ad(\widetilde{G})^{\ast}(0,0,0,0,-2k)$-subordinate subalgebra inducing an integral point $(\bar{\lambda},H^0)$, there exist $4k^2$ periodizing maps $\Theta_k^j:L^2(H_{\lambda}\backslash\widetilde{G})\rightarrow L_k^2(P),~j=1,\dots,4k^2$ which achieve an orthogonal decomposition
\[
L_k^2(P)\simeq\bigoplus_{j=1}^{4k^2}\Theta_k^j(L^2(H_{\lambda}\backslash\widetilde{G})) \simeq 4k^2 L^2(H_{\lambda}\backslash\widetilde{G})
\]
of the $k$-isotypical subspace $L_k^2(P)$ into irreducible representations of $\widetilde{G}$.
\end{corollary}

\subparagraph{Example}
For each $k\in\mathbb{Z}\setminus\{0\}$ and each $m,n=0,1,\dots,2k-1$, the periodizing map associated to the integral point $(\bar{\lambda}_k^{m,n},H^0)$ (Lemma \ref{lemma:integral points}) is: for\footnote{Since we have identified  $H^0\backslash\widetilde{G}\simeq\mathbb{R}^2$ via the section $H^0(x,0,0,t,0)\mapsto(x,0,0,t,0),$ the coset $H^0\in H^0\backslash\widetilde{G}$ corresponds to the point $(0,0)\in\mathbb{R}^2.$} $f\in L^2(\mathbb{R}^2)$
\begin{multline}
\label{eqn:e0 periodizing maps}
\left(\Theta_k^{m,n}f\right)(x,y,z,t,u) =\sum_{a,b\in\mathbb{Z}}\left[\pi_{-2k}^{m,n}((a,0,0,b,0)\cdot(x,y,z,t,u))f\right](0,0)\\
=e^{-2\pi i(my-n(z+xy))}e^{-4\pi ik(u-zx)}\sum_{a,b\in\mathbb{Z}}e^{2\pi inya}e^{-4\pi  ik(by-za-\frac{y}2(x+a)^2)}f(x+a,t+b).
\end{multline}
\hfill$\square$

\subsection{\label{subsec:trans rules} Transformation rules}

The periodizing maps $\Theta_k^j$ are constructed so that the the resulting function is  equivalent to a section of the nontrivializable line bundle $\ell^{\otimes k}$. Hence, when $\ell$ is lifted to a trivializable line bundle $\check{\ell}\rightarrow G$ and then trivialized, the function which corresponds to $\Theta_k^j f$ is pseudoperiodic, that is, the functions $\vartheta_k^j f$ satisfy transformation rules associated to the integral lattice $\Gamma_0$.

\subparagraph{Remark}
In the classical theory, there is another aspect of the pseudoperiodicity of $\vartheta$-functions: polarizations (complex structures); the classical $\vartheta$-functions are \textit{holomorphic} sections of a line bundle over the torus. Different trivializations of the lifted line bundle express the covariant notion of holomorphic differently. For example, in \eqref{eqn:periodized theta}, the line bundle  $\check{\ell}\rightarrow\mathbb{R}^2$ was trivialized in such a way that a holomorphic  section takes the form $f(z)e^{-\pi y^2}.$ In the current situation, there is no relevant complex structure (polarization) with respect to which our $\vartheta$-functions will be holomorphic. \hfill$\square$

\medskip
Let $\tilde{s}\in L_k^2(P)$. By definition, $\ell=P\times_{\rho}\mathbb{C}=(\Gamma_k\backslash K)\times_{\rho}\mathbb{C}$, and $\tilde{s}$ determines a section $s\in L^2(M,\ell^{\otimes k})$ by the correspondence\footnote{The equivariance of $\tilde{s}$, combined with the definition of the  equivalence class (see Section \ref{subsec:prequantum bundles}), insures that the correspondence is well-defined (i.e., independent of choice of $[v]$).}
\[
s(\Gamma_0\mathbf{g})=[(\Gamma_k(\mathbf{g},[v]),\tilde{s}(\Gamma_k(\mathbf{g},[v])))].
\]
This section induces a section $\check{s}\in\Gamma(\check{\ell}^{\otimes k}=K\times_{\rho^{(k)}}\mathbb{C})$ given by
\[
\check{s}(\mathbf{g})=[(\mathbf{g},[v]),\tilde{s}(\Gamma_k(\mathbf{g},[v]))].
\]

Now, there are many reasonable ways to trivialize $\check{\ell}$. For example, one could use  the global section $s_0(\mathbf{g})=[(\mathbf{g},[0]),1].$ An approach which is common to  geometric quantization is to choose a global symplectic potential $\theta$ (which trivializes $\check{\ell}$ in a standard way). Yet another approach would be to define an action of $G$ on  $\check {\ell}$ and use it to map $\check{\ell}_{\mathbf{g}}\rightarrow\check{\ell}_1\simeq\mathbb{C}$.

We will take the first approach because it is the simplest and the particular trivialization we choose is basically irrelevant for our purposes. In this trivialization, the function on $G$ associated to the function $\tilde{s}\in L_k^2(P)$ is
\[
\mathbf{g}\mapsto\tilde{s}(\Gamma_k(\mathbf{g},[0])).
\]

Hence, the function $\vartheta_k^jf\in L^2(FD_{\Gamma_0\backslash G})$ associated to $\Theta_k^jf\in L_k^2(P)$ is
\begin{align*}
(\vartheta_k^j f)(\mathbf{g})  & =(\Theta_k^j f)(\Gamma_k(\mathbf{g},[0]))\\
& =\sum_{[\gamma]\in(\Gamma\cap H_k)\backslash\Gamma}\left[\operatorname{Ind}_{H_k}^{\widetilde{G}}([\gamma](\mathbf{g},[0]))f\right](H_k).
\end{align*}
We can now state the pseudoperiodicity of the images $\vartheta_k^j f$.

\begin{theorem}
\label{thm:pseudoperiodicity}
Let $\gamma_0\in\Gamma_0.$ Then
\[
(\vartheta_k^jf)(\gamma_0\mathbf{g})=\exp\{2\pi ik\psi(\gamma_0^{-1},\mathbf{g})\}(\vartheta_k^jf)(\mathbf{g}).
\]
where $\psi(\tilde{g}_1,\tilde{g}_2)$ is defined by the group
multiplication \eqref{eqn: grp law scalar}.
\end{theorem}

\begin{proof}
First, observe that
\begin{equation}
\label{eqn:transproof1}
(\gamma_0^{-1},[0])\cdot(\gamma_0\mathbf{g},[0]) =(\mathbf{g},\left[\psi(\gamma_0^{-1},\mathbf{g})\right]) =(\mathbf{g},[0])\cdot(\mathbf{0},\left[  \psi(\gamma_0^{-1},\mathbf{g})\right]).
\end{equation}
(Recall that we write $\mathbf{0}=(0,0,0,0)\in G$.) Then
\begin{align*}
(\vartheta_k^jf)(\gamma_0\mathbf{g}) &= (\Theta_k^jf)(\Gamma_k(\gamma_0\mathbf{g},[0]))\\
& =\sum_{[\gamma]\in(\Gamma_k\cap H_k)\backslash\Gamma_k} \left[\operatorname{Ind}_{H_k}^{\widetilde{G}}([\gamma](\gamma_0\mathbf{g},[0]))f\right](H_k)\\
& =\sum\bar{\lambda}_k(h(H_k,[\gamma](\gamma_0\mathbf{g},[0])))\,f(H_k[\gamma](\gamma_0\mathbf{g},[0])).
\end{align*}
Now, let $\tilde{\gamma}_0^{-1}:=(\gamma_0^{-1},[0])\in\Gamma_k.$ Then $\tilde{\gamma}_0^{-1}\cdot(\gamma_0\mathbf{g},[0]) =(\mathbf{g},\left[0\right])\cdot(\mathbf{0},\left[\psi(\gamma_0^{-1},\mathbf{g})\right])$. Relabel the sum $[\gamma]\mapsto[\gamma]\tilde{\gamma}_0^{-1}$. Using equation \eqref{eqn:transproof1}, the above equation becomes
\[
(\vartheta_k^jf)(\gamma_0\mathbf{g}) =\sum\bar{\lambda}_k(h(H_k,[\gamma](\mathbf{g},[0])\cdot(\mathbf{0},\left[\psi(\gamma_0^{-1},\mathbf{g})\right]))) \,f(H_k[\gamma](\mathbf{g},[0])\cdot(\mathbf{0},\left[  \psi(\gamma_0^{-1},\mathbf{g})\right])).
\]
The cocycle property \eqref{eqn:cocycle} of $h$ and the observation that $(\mathbf{0},\left[\psi(\gamma_0^{-1},\mathbf{g})\right])$ is central and in $H_k$ then imply
\begin{align*}
h(H_k,[\gamma](\mathbf{g},[0])(\mathbf{0},\left[\psi(\gamma_0^{-1},\mathbf{g})\right])) &= h(H_k,(\mathbf{0},\psi(\gamma_0^{-1},\mathbf{g}))[\gamma](\mathbf{g},[0]))\\
& =h(H_k,[\gamma](\mathbf{g},[0]))h(H_k,(\mathbf{0},\psi(\gamma_0^{-1},\mathbf{g}))).
\end{align*}
Again using that $(\mathbf{0},\psi(\gamma_0^{-1},\mathbf{g}))$ is central and in $H_k,$ we obtain
\begin{align*}
(\vartheta_k^jf)(\gamma_0\mathbf{g})& =\sum\bar{\lambda}_k(h(H_k,[\gamma](\mathbf{g},[0]))) \bar{\lambda}_k(h(H_k,(\mathbf{0},\left[\psi(\gamma_0^{-1},\mathbf{g})\right])))\,f(H_k[\gamma](\mathbf{g},[0]))\\
&=\bar{\lambda}_k(h(H_k,(\mathbf{0},\left[\psi(\gamma_0^{-1},\mathbf{g})\right])))(\vartheta_k^jf)(\mathbf{g}).
\end{align*}

The final step is to simplify the first term above: recall that $h$ is defined by
\[
h(x,g)=s(x)\cdot g\cdot s(x\cdot g)^{-1}
\]
for some section $s:H_k\backslash\widetilde{G}\rightarrow\widetilde{G}$ which we assume  normalized so that $s(H_k)=\mathbf{0}\in\widetilde{G}$ (\textit{ff.} \eqref{eqn:master eqn}). In particular, for $g\in H_k,$
\[
h(H_k,g)=s(H_k)\cdot g\cdot s(H_k)^{-1}=g,
\]
whence
\[
\bar{\lambda}_k(h(H_k,(\mathbf{0},\psi(\gamma_0^{-1},\mathbf{g}))))=\exp\{4\pi ik\psi(\gamma_0^{-1},\mathbf{g})\}
\]
as desired.
\end{proof}

\smallskip
\subparagraph{Example}
Each of the periodizing maps $\Theta_k^{m,n}$ \eqref{eqn:e0 periodizing maps} constructed from the subordinate subalgebra $\mathfrak{h}^{e=0}$ yields a map $\vartheta_k^{m,n}:L^2(\mathbb{R}^2)\rightarrow L^2(FD_{\Gamma_0\backslash G})$ given by
\[
(\vartheta_k^{m,n}f)(x,y,z,t)=e^{-2\pi i[my-n(z+xy)]}e^{-4\pi ikzx}\sum_{a,b\in\mathbb{Z}}e^{2\pi inya}e^{-4\pi ik(by-za-\frac{y}2(x+a)^2)}f(x+a,t+b).
\]
By Theorem \ref{thm:pseudoperiodicity}, for each $f\in L^2(\mathbb{R}^2)$ the functions $\vartheta_k^{m,n}f$ satisfy the pseudoperiodicity conditions
\begin{align*}
(\vartheta_k^{m,n}f)(x+1,y,z,t) &=(\vartheta_k^{m,n}f)(x,y,z,t),\\(\vartheta_k^{m,n}f)(x,y+1,z-x,t) &=e^{-2\pi ikx^2}(\vartheta_k^{m,n}f)(x,y,z,t),\\
(\vartheta_k^{m,n}f)(x,y,z+1,t) &  =e^{4\pi ikx}(\vartheta_k^{m,n}f)(x,y,z,t),\text{ and}\\
(\vartheta_k^{m,n}f)(x,y,z,t+1) &  =e^{4\pi iky}(\vartheta_k^{m,n}f)(x,y,z,t)
\end{align*}
(this can also be easily checked by direct calculation). These are the pseudoperiodicity  conditions given in the Introduction.\hfill$\square$

\section{Harmonic analysis on P\label{sec:harmonic on P}}

We face the problem of computing the spectrum of the Laplacian on $M$ acting on the $k$-th tensor power $\ell^{\otimes k}$ of the prequantum line bundle associated to $P$. Although we do not obtain an exact description of the spectrum, a semiclassical analysis proves to be sufficient for our purposes. In this section, we will describe the Laplacian on $M$ acting on sections of $\ell^{\otimes k}$ and hence, with our usual identification, acting on $k$-equivariant functions on $P$. We will use the quantum Birkhoff canonical form of this Laplacian to deduce certain semiclassical properties and hence the structure of the almost K\"{a}hler quantization of $M$.

The almost K\"{a}hler quantization of $M$ is defined to be the approximate kernel of a rescaled metric Laplacian. In the classical case, this is the vector space of holomorphic sections of the prequantum bundle. Because of their holomorphicity, these sections are completely determined by their pseudoperiodicity. Here, there does not exist any complex structure with respect to which the sections in the almost K\"{a}hler quantization of $M$ are holomorphic. Consequently, we cannot reconstruct them from their pseudoperiodicity alone; we are forced to try to solve for the approximate kernel directly.

As we have done throughout this paper, we identify a section $s\in L^2(M,\ell^{\otimes k})$ with a $k$-equivariant function (Lemma \ref{lemma:identification}) $\tilde{s}\in L_k^2(P)$ in the standard way, i.e.,
\[
s(x)=[(p,\tilde{s}(p))]
\]
for $\pi(p)=x$, where $\tilde{s}$ is $k$-equivariant if $\tilde{s}(p\cdot e^{2\pi i\theta}) =e^{-4\pi ik\theta}\tilde{s}(p)$. We will find the computations are simpler when stated in terms of $L_k^2(P)$.

As we will see in Section \ref{subsec:Laplacians}, the Laplacian on $M$ acting on $\ell^{\otimes k}$ can be written in terms of the standard Euclidean Laplacian $\Delta_E$ acting on $P$. Recall our left-invariant metric\footnote{$\{\beta_{L}^j\}_{j=1,2,3,T,U}$ is the left-invariant coframe which is dual to $\{X_1,X_2,X_3,T,U\}$ at the origin.} on $\widetilde{G}$ and hence on $P=\widetilde{\Gamma}\backslash\widetilde{G}$:
\begin{equation}
\label{eqn:linv metric}
g=(\beta_{L}^1)^2+(\beta_{L}^2)^2+(\beta_{L}^3)^2+(\beta_{L}^T)^2+(\beta_{L}^U)^2.
\end{equation}
Since right translation is generated by the left-invariant vector fields, the Euclidean Laplacian on $P$ is given by
\[
\Delta_E =-\sum_{j=1}^3[\rho_{\ast}(X_{j})]^2 -[\rho_{\ast}(T)]^2-[\rho_{\ast}(U)]^2,
\]
where $\rho$ is the \emph{right} regular representation of $\widetilde{G}$ on $L^2(P)$, which is given by $(\rho(g)f)(x)=f(xg)$.

The right action induces a representation of $\widetilde{G}$ on $L^2(P)$ which commutes with $\Delta_E.$ Hence, $\Delta_E$ preserves $\widetilde{G}$-invariant subspaces, and we can study the harmonic analysis of $\Delta_E$ by its pullback action on the representation spaces of $\widetilde{G}$.

\subsection{Laplacians\label{subsec:Laplacians}}

In the K\"{a}hler case, the Hodge Laplacian is equal to a rescaled metric Laplacian. Here, since the Kodaira--Thurston manifold does not admit any (positive) K\"{a}hler structure, any Hodge Laplacian will be badly behaved. But we can still write the metric (and rescaled metric) Laplacian (on $M$) acting on the $k$-th tensor power of the prequantum bundle.

We have chosen a left-invariant metric on $G$ defined by
\[
g=(\beta_{L}^1)^2+(\beta_{L}^2)^2+(\beta_{L}^3)^2+(\beta_{L}^T)^2.
\]
Since $g$ is left-invariant, it descends to a metric, denoted also by $g$, on $\Gamma_0\backslash G$.

The connection on $P:=\widetilde{\Gamma}\backslash\widetilde{G}$ defined by the connection $1$-form $2\pi\beta^U$ induces a connection on $\ell^{\otimes k}$ and hence a covariant derivative acting on sections of $\ell^{\otimes k}.$ The corresponding covariant derivative on $L_k^2(P)$ is (see \cite[\textit{pp.} 22]{BGV}, for example)
\[
\widetilde{\nabla}=d-4\pi k\beta_{L}^U
\]
since $\rho(e^{2\pi i\theta})=e^{2\pi i(-2\theta)}.$ The coframe $\{\beta_{L}^1,\beta_{L}^2,\beta_{L}^3,\beta_{L}^T\}$ is dual to $\{X_1^{L},X_2^{L},X_3^{L},T^{L}\},$ so we immediately have
\[
\widetilde{\nabla}_{X_1^{L}}=X_1^{L},\widetilde{\nabla}_{X_2^{L}} =X_2^{L},~\widetilde{\nabla}_{X_3^{L}}=X_3^{L},~\widetilde{\nabla}_{T^{L}}=T^{L}.
\]
The left-invariant frame $\{X_1^{L},X_2^{L},X_3^{L},T^{L},U^{L}\}$ is given by
\[
\begin{split}
X_1^{L}=\frac{\partial}{\partial a^1}-a^2\frac{\partial}{\partial a^3}&+a^3\frac{\partial}{\partial v},\ X_2^{L}= \frac{\partial}{\partial a^2}+r\frac{\partial}{\partial v},\ X_3^{L}=\frac{\partial}{\partial a^3},\\
&T^{L}=\frac{\partial}{\partial r},\ U^{L}=\frac{\partial}{\partial v}.
\end{split}
\]
Hence, the metric Laplacian acting on $k$-equivariant functions on the prequantum circle bundle is
\begin{align*}
\Delta^{(k)}&=-\left[\left(X_1^{L}\right)^2 +\left(X_2^{L}\right)^2 +\left(X_3^{L}\right)^2 +\left(T^{L}\right)^2\right]\\
&=-\left[(\partial_{a^1}-a^1\partial_{a^3} +a^3\partial_{v})^2 +(\partial_{a^2}+r\partial_{v})^2 +\partial_{a^3}^2+\partial_{r}^2\right].
\end{align*}
The rescaled metric Laplacian acting on $k$-equivariant functions on $P$ (which, if $M$ where K\"{a}hler, would be equal to the Hodge Laplacian) is then
\[
\Delta_{\bullet}^{(k)}:=\Delta^{(k)}-\frac12\dim(M)\cdot2\pi k=\Delta^{(k)}-4\pi k.
\]

Associated to the metric \eqref{eqn:linv metric} is the Euclidean (i.e., standard) Laplacian acting on $P$:
\[
\Delta_E=-\left[\left(X_1^{L}\right)^2+\left(X_2^{L}\right)^2 +\left(X_3^{L}\right)^2 +\left(T^{L}\right)^2 +\left(U^{L}\right)^2\right].
\]
Using the fact that, when applied to a $k$-equivariant function, $\partial_{v}=-4\pi ik$, we see that the three Laplacians are related by
\[
\Delta_{\bullet}^{(k)}=\Delta^{(k)}-8\pi k=(\Delta_E-16\pi^2k^2)-8\pi k.
\]

Given a periodizing map $\Theta_k:L^2(H_k\backslash\widetilde{G})\rightarrow L_k^2(\Gamma\backslash\widetilde{G})$ associated to an integral point $(\bar{\lambda},H)$ of an orbit $Ad(\widetilde{G})^{\ast}(0,0,0,0,-2k),~k\in\mathbb{Z}\setminus\{0\}$, we define the filtered Laplacian $\Delta_k\in\mathcal{O}\left(L^2(H_k\backslash\widetilde{G})\right)$ by
\[
\Delta_k=\Theta_k^{-1}\Delta^{(k)}\Theta_k.
\]
Since $\Theta_k$ intertwines the $\widetilde{G}$-action, we see that
\[
\Delta_k=-\left[\left(\left(\pi_{-2k}\right)_{\ast}(X_1)\right)^2 +\left(\left(\pi_{-2k}\right)_{\ast}(X_2)\right)^2 +\left(\left(\pi_{-2k}\right)_{\ast}(X_3)\right)^2+\left(\left(\pi_{-2k}\right)_{\ast}(T)\right)^2\right]
\]
where $\left[(\pi_{-2k})_{\ast}(X)f\right]([g]) :=\left.\frac{d}{dt}\right\vert_{t=0}\left(\pi_{-2k}\left(e^{tX}\right)\right) f([g])$.

We will use the representation $\pi_{-2k}^{0,0}$ of Section \ref{sec:rep thry ii} and its associated periodizing map $\Theta_k^{0,0}$ to compute the filtered Laplacian. The result is
\[
\Delta_k=-\partial_{xx}-\partial_{tt}+16k^2\pi^2\left[(x^2+t^2)+x^2(\frac{x^2}4-t)\right].
\]
The Laplacian $-\partial_{xx}-\partial_{tt}$ is a nonnegative operator, and therefore
\[
\langle\Delta_kf,f\rangle\geq\langle Vf,f\rangle\geq0,\quad\forall f\in C_0^{\infty}(\mathbb{R}^2).
\]
Hence, the spectrum of $\Delta_k$ is nonnegative.

The metric Laplacian $\Delta^{(k)}$ commutes with the right action of $\widetilde{G}$ on $P$, and hence preserves any decomposition of $L^2(P)$ into invariant subspaces. In particular, for each $k\in\mathbb{Z}\setminus\{0\}$ and each choice of representatives of the orbits $\left.\left[(\bar{\lambda},H_{\lambda}) \cdot\widetilde{G}\right]_{\mathbb{Z}}\right/ \widetilde{\Gamma},$ there exist periodizing maps $\Theta_k^j,~j=1,\dots,4k^2$ whose images are orthogonal irreducible subspaces of $L_k^2(P)$. Indeed, each $\Theta_k^j$ identifies an irreducible subspace with $L^2(H_k\backslash\widetilde{G})$, and under this identification, the restriction of $\Delta^{(k)}$ to the irreducible subspace acts as $\Delta_k.$ We have therefore proved the following.

\begin{theorem}
For each $k\in\mathbb{Z}\setminus\{0\}$, the spectrum of the metric Laplacian on $M$ acting on sections of $k$-th tensor power $\ell^{\otimes k} $ is equal to the spectrum of $\Delta_k$, repeated with multiplicity $4k^2.$
\end{theorem}

\subsection{Almost K\"{a}hler quantization of $M$}

In order to study the spectrum of the family of operators $\Delta_k,$ we introduce a formal deformation parameter. In geometric quantization, the tensor power of the prequantum line bundle is interpreted as $1/4\pi\hbar$, that is,
\[
4\pi k=1/\hbar.
\]

The work of Charles and Vu Ngoc \cite{Charles-Vu_Ngoc} yields estimates on the spectrum of $\Delta_k$ from the quantum Birkhoff normal form of $\Delta_{1/\hbar}$ for small $\hbar$; in particular, the estimates will hold for $k$ sufficiently large (i.e. in the semiclassical limit). The main result is that the spectrum of $\Delta_k$ is an order $\hbar^2$ correction to the spectrum of the simple harmonic oscillator, that is, there are spectral bands around each eigenvalue of the simple harmonic oscillator whose widths are order $\hbar^2$. The separation of the eigenvalues of the simple harmonic oscillator, on the other hand, is order $\hbar.$ Hence, the separation between the lowest spectral bands of $\Delta_k$ is order $\hbar$---this is the simple verification of the expected spectral band gap.

In this section, we will find it useful to use a certain conjugation of our filtered Laplacian; let $\varepsilon=\sqrt{\hbar}$ and $U:L^2(\mathbb{R}^2)\rightarrow L^2(\mathbb{R}^2)$ be the unitary map $U(f)(x)=\hbar^{1/4}f(\sqrt{\hbar}x).$ Then we define\footnote{This transformation is natural for semiclassical analysis; for example, one way to compute the semiclassical asymptotics of $\int e^{-x^2/\hbar}f(x)dx $ is to begin with the change of variables $x\mapsto x/\sqrt{\hbar}$.}
\begin{align*}
H  &=\hbar U\Delta_{1/\hbar}U^{-1}\\
&=-(\partial_{x}^2+\partial t) +x^2+t^2 +\varepsilon\left(x^2t\right) +\varepsilon^2\left(\frac{x^4}4\right).
\end{align*}
In this form, it is clear that $\Delta_k$ can be regarded as a perturbation of the simple harmonic oscillator.

As is usually the case when dealing with the simple harmonic oscillator, computations are greatly simplified by the introduction of ladder operators. Let $x_1=x$, $x_2=t$, and define
\[
a_{i}=\frac{\partial_{x_{i}}+x_{i}}{\sqrt2},\ b_{i}:=a_{i}^{\ast} =\frac{-\partial_{x_{i}}+x_{i}}{\sqrt2},~i=1,2.
\]
The standard commutators are then $[a_{i},b_{j}]=\delta_{ij}$ and $[a_{i},a_{j}]=[b_{i},b_{j}]=0.$

We now recall the Birkhoff canonical form (see \cite{Charles-Vu_Ngoc} for details). Consider the graded algebra of differential operators $\mathcal{D}[[\varepsilon]]:=\bigoplus\nolimits_{j=2}^{\infty}\varepsilon^{j-2}\mathcal{D}_{j}$ where\footnote{We use standard multi-index notation.}
\[
\mathcal{D}_{j}=\left\{\sum\limits_{\substack{k\leq j \\k\equiv j\operatorname{mod}\,2}} \sum\limits_{\left\vert\alpha\right\vert +\left\vert\beta\right\vert =k} c_{\alpha\beta}a^{\alpha}b^{\beta}\right\}.
\]
A convenient basis for $\mathcal{D}_{j}$ is $\{a^{\alpha}b^{\beta}:\left\vert\alpha\right\vert +\left\vert \beta\right\vert \leq,\equiv j\operatorname{mod}\,2\}$, since
\[
H_2=\sum_{i=1,2}\left(  a_{i}b_{i}-\frac12\right)
\]
and
\begin{equation}
\label{eqn:diag basis}
[a_{i}b_{i},a^{\alpha}b^{\beta}]=(\beta_{i}-\alpha_{i})a^{\alpha}b^{\beta},
\end{equation}
imply that $H_2$---the simple harmonic oscillator---is diagonal:
\[
[H_2,a^{\alpha}b^{\beta}]=(\left\vert \beta\right\vert -\left\vert\alpha\right\vert )a^{\alpha}b^{\beta}.
\]

The grading of $H$ is given by $H=H_2+\varepsilon H_3+\varepsilon^2H_4,$ where
\begin{align*}
H_3  & =\frac1{2\sqrt2}(a_1+b_1)^2(a_2+b_2),\text{ and}\\
H_4  & =\frac1{16}(a_1+b_1)^4.
\end{align*}

Let $\operatorname{ad}_{A}(\cdot)=[A,\cdot]$ for $A\in\mathcal{D}[[\varepsilon]].$ Each $\mathcal{D}_{j}$ can be decomposed as
\begin{equation}
\label{eqn:H2directsum}
\mathcal{D}_{j}=\ker\left.\operatorname{ad}_{H_2}\right\vert_{\mathcal{D}_{j}} \oplus\operatorname{im}\left. \operatorname{ad}_{H_2}\right\vert_{\mathcal{D}_{j}}.
\end{equation}
The following is an easy computation using \eqref{eqn:diag basis}.

\begin{lemma}
We have $\ker\operatorname{ad}_{H_2}=\operatorname{span}\{a^{\alpha}b^{\beta}:\left\vert \alpha\right\vert =\left\vert \beta\right\vert \}$ and $\ker\left.\operatorname{ad}_{H_2}\right\vert_{\mathcal{D}_{j}}=\{0\}$ if and only if $j$ is odd.
\end{lemma}

The quantum Birkhoff normal form is summarized in the following theorem.

\begin{theorem}
There exist $A(\varepsilon),K(\varepsilon)\in\mathcal{D}[[\varepsilon]]$ such that
\[
\exp(ad(A(\varepsilon)))H(\varepsilon)=K(\varepsilon),
\]
where $A(\varepsilon)=\varepsilon A_3+\varepsilon^2A_4+\cdots$, and $K(\varepsilon)=H_2+\varepsilon K_3+\cdots$ is such that $\operatorname{ad}(H_2)K_{j}=0,~j=3,4,\dots$, that is, $K_{j}\in\ker\left.\,\operatorname{ad}_{H_2}\right\vert _{\mathcal{D}_{j}}.$
\end{theorem}

It is possible to compute the terms $A_{j}$ and $K_{j}$ inductively\footnote{Expanding and matching terms, one sees that we must choose $A_{j},~j=3,4,...$ so that
\begin{align*}
A_2  & =0,~K_2=H_2,\\
K_3  & =H_3+[A_3,H_2]\in\left.\ker\,\operatorname{ad}(H_2)\right\vert _{\mathcal{D}_3},\\
K_4  & =H_4+[A_3,H_3]+\frac12[A_3,[A_3,H_2]] +[A_4,H_2]\in\left.  \ker\,\operatorname{ad}(H_2)\right\vert_{\mathcal{D}_4},
\end{align*}
\par
Indeed, we see that at each step we must write
\[
[H_2,A_{j}]+K_{j}=H_{j}+....
\]
which is possible because of \eqref{eqn:H2directsum}. Hence, we can find $K_{j}$ by computing
\[
K_{j}=\operatorname{proj}\nolimits_{\left.\ker\operatorname{ad}(H_2)\right \vert _{\mathcal{D}_{j}}}(H_{j}+...).
\]
Then, to find $A_{j},$ compute
\[
A_{j}=\operatorname{ad}(H_2)^{-1}(H_{j}+...-K_{j}),
\]
which, since $\operatorname{ad}(H_2)$ is diagonal in our basis of ladder operators, is straightforward.}. The first few are
\begin{align*}
K_2  & =H_2,~A_2=0,\\
K_3  & =0,~~A_3=\frac1{2\sqrt2}\left(  -\frac13a_1^2 a_2-a_1^2b_2+b_1^2a_2+\frac13b_1^2b_2-2a_1b_1 a_2+2a_1b_1b_2+a_2-b_2\right),\\
K_4  & =\frac1{24}\left(-\frac12+10a_1b_1+8a_2b_2-a_1^2b_1^2-12a_1^2b_2^2-16a_1a_2b_1b_2-12a_2^2b_1^2\right) \\
A_4  & =\frac1{192}\left(  -4a_1^2-16a_2^2+4b_1^2+16b_2^2-5a_1^4-8a_1^3b_1-8a_1^2a_2^2+32a_1^2a_2b_2 +32a_1a_2^2b_1\right. \\
& \left.+8a_1b_1^3-32a_1b_1b_2^2-32a_2b_1^2b_2+5b_1^4+8b_1^2b_2^2\right).
\end{align*}

The utility of the quantum Birkhoff normal form for us is a result of Charles and Vu Ngoc in \cite{Charles-Vu_Ngoc} which says the spectrum of $\Delta_k$ is a perturbation of the spectrum of $H_2$. In particular, around each eigenvalue of $H_2$ there is a spectral band of $\Delta_k$ whose width is $O(\hbar)$ (for large $\varepsilon=\sqrt{\hbar}$, these spectral bands widen and eventually overlap, but we are mainly interested in the lowest band, centered at $1$). Charles and Vu Ngoc prove the following theorem.

\begin{theorem}
\label{thm:C-S spectrum}
There exists $\varepsilon_0>0$ and $C>0$ such that for each $\varepsilon\in(0,\varepsilon_0]$
\[
\operatorname{spec}(\Delta_k)\cap(-\infty,C\varepsilon)\subset \bigcup\limits_{E_{N}\in\operatorname{spec}(H_2)} [E_{N}-\frac{\varepsilon^2}3,E_{N}+\frac{\varepsilon^2}3].
\]
\end{theorem}

In our case, though, since $K_3=0$, the width of the spectral bands is $O(\varepsilon^4=\hbar^2)$ (that is, the Birkhoff canonical form of our operator is an $O(\varepsilon^4)$ correction). Since the separation of the eigenvalues of the harmonic oscillator is $O(\varepsilon^2),$ we see that as $\varepsilon\rightarrow0$, a spectral gap of width $O(\varepsilon^2)$ appears between the ground state band (centered at $1$) and the first excited band. This is the direct verification of the spectral band gap described in Theorem \ref{thm:GU}.

\subparagraph{Remark}
Although it is not relevant to the almost K\"{a}hler quantization of the Kodaira--Thurston manifold, we note that the spectrum of the metric Laplacian on $M$ acting on functions (i.e., the $k=0$ case) can be computed exactly since the filtered Laplacians for the functional dimension-$0$ and -$1$ representations can be inverted explicitly.\hfill$\square$

\smallskip
The almost K\"{a}hler quantization of the Kodaira--Thurston manifold $M$ is defined to be the $\mathbb{C}$-span of the set of low-lying eigenstates of the rescaled metric Laplacian $\Delta_{\bullet}^{(k)}$ which acts on sections of the $k$-th tensor power $\ell^{\otimes k}$ of the prequantum line bundle. The dimension of this space is, for $k$ sufficiently large, the Riemann--Roch number of $M$ twisted by $\ell^{\otimes k};$ a routine computation shows that this Riemann--Roch number is $4k^2.$ As we have seen in Section \ref{sec:harmonic on P}, the rescaled Laplacian $\Delta_{\bullet}^{(k)}$ decomposes as a direct sum of $4k^2$ copies of the filtered Laplacian $\Delta_k$ acting on $L^2(\mathbb{R}^2)$. We have therefore proved that:

\begin{corollary}
The rescaled filtered Laplacian $\Delta_k-4\pi k$, for $k$ sufficiently large, has a unique ground state which separates from the excited spectrum by a gap of order $k.$
\end{corollary}

It then follows that if $\psi_0$ denotes the unique ground state of $\Delta_k-4\pi k,$ the almost K\"{a}hler quantization of $M$, at level $4\pi k=1/\hbar$, consists of the images of $\psi_0$ under the periodizing maps (for any choice of subordinate subalgebra), that is,
\[
\mathcal{H}_{M}^{(k)}:=\operatorname*{span}\nolimits_{\mathbb{C}}\{ \Theta_k^1\psi_0,\Theta_k^2\psi_0,...,\Theta_k^{4k^2}\psi_0\}.
\]

\section{Appendix: Faithful matrix representations}

For computational convenience, we record here faithful matrix representations of the Lie groups and algebras studied in this paper. We begin with the product $G=\operatorname{Heis}(3)\times\mathbb{R}$ of the three-dimensional Heisenberg group with $\mathbb{R},$ which we realize as the group of $5\times5$ matrices of the form
\[
[ a^1,a^2,a^3,r]=
\begin{pmatrix}
1 & a^1 & a^2 & 2a^3+a^1a^2 & 0\\
0 & 1 & 0 & a^2 & 0\\
0 & 0 & 1 & -a^1 & 0\\
0 & 0 & 0 & 1 & 0\\
0 & 0 & 0 & 0 & e^{r}
\end{pmatrix}.
\]
The group law \eqref{eqn:G group law} is then obtained from the usual matrix product.

A basis for $\mathfrak{h}=Lie(\operatorname{Heis}(3))$ is
\[
X_1=
\begin{pmatrix}
0 & 1 & 0 & 0\\
0 & 0 & 0 & 0\\
0 & 0 & 0 & -1\\
0 & 0 & 0 & 0
\end{pmatrix}
\qquad X_2=
\begin{pmatrix}
0 & 0 & 1 & 0\\
0 & 0 & 0 & 1\\
0 & 0 & 0 & 0\\
0 & 0 & 0 & 0
\end{pmatrix}
\qquad X_3 =
\begin{pmatrix}
0 & 0 & 0 & 2\\
0 & 0 & 0 & 0\\
0 & 0 & 0 & 0\\
0 & 0 & 0 & 0
\end{pmatrix}.
\]
These satisfy $[X_1,X_2]=X_3 .$ The canonical coordinates on $G$ are then expressed in terms of the matrix exponential as
\[
[a^1,a^2,a^3,r]=\exp(a^1 X_1)\exp(a^2 X_2)\exp(a^3X_3)\oplus e^{r}.
\]

Next, a matrix representation of the Lie algebra $\tilde{\mathfrak{g}}$ of the central extension $\widetilde{G}$:
\begin{gather*}
X_1=
\begin{pmatrix}
0 & 1 & 0 & 0 & 0\\
0 & 0 & 0 & 0 & 0\\
0 & 0 & 0 & -1 & 0\\
0 & 0 & 0 & 0 & 1\\
0 & 0 & 0 & 0 & 0
\end{pmatrix}
,\ X_2=
\begin{pmatrix}
0 & 0 & 1 & 0 & 0\\
0 & 0 & 0 & 1 & 0\\
0 & 0 & 0 & 0 & 1\\
0 & 0 & 0 & 0 & 0\\
0 & 0 & 0 & 0 & 0
\end{pmatrix}
,~X_3=
\begin{pmatrix}
0 & 0 & 0 & 2 & 0\\
0 & 0 & 0 & 0 & -1\\
0 & 0 & 0 & 0 & 0\\
0 & 0 & 0 & 0 & 0\\
0 & 0 & 0 & 0 & 0
\end{pmatrix}
,\ \\
T=%
\begin{pmatrix}
0 & 0 & 0 & 0 & 0\\
0 & 0 & 0 & 0 & 0\\
0 & 0 & 0 & 0 & -3\\
0 & 0 & 0 & 0 & 0\\
0 & 0 & 0 & 0 & 0
\end{pmatrix}
,U=
\begin{pmatrix}
0 & 0 & 0 & 0 & 3\\
0 & 0 & 0 & 0 & 0\\
0 & 0 & 0 & 0 & 0\\
0 & 0 & 0 & 0 & 0\\
0 & 0 & 0 & 0 & 0
\end{pmatrix}.
\end{gather*}
Again, the canonical coordinates on $\widetilde{G}$ can be expressed, using the matrix exponential, in terms of the above matrices:
\begin{align*}
[ a^1,a^2,a^3,r,v] &  =\exp(a^1X_1)\exp(a^2X_2)\exp(a^3X_3)\exp(rT)\exp(vU)\\
&  =
\begin{pmatrix}
1 & a^1 & a^2 & 2a^3+a^1a^2 & 3v-3ra^2+{\tfrac12}(a^2)^2-a^1a^3\\
0 & 1 & 0 & a^2 & -a^3\\
0 & 0 & 1 & -a^1 & -3r+{\tfrac12}(a^1)^2+a^2\\
0 & 0 & 0 & 1 & a^1\\
0 & 0 & 0 & 0 & 1
\end{pmatrix}.
\end{align*}
The group law \eqref{eqn:gt grp law} can be worked out explicitly using the
above matrices.

\providecommand{\bysame}{\leavevmode\hbox to3em{\hrulefill}\thinspace}

\end{document}